\documentclass[english]{amsart}
\usepackage[T1]{fontenc}
\usepackage[latin9]{inputenc}
\usepackage{amstext} 
\usepackage{amsthm}
\usepackage{amssymb}
\usepackage{amsfonts}

\makeatletter
\numberwithin{equation}{section}
\numberwithin{figure}{section}
\theoremstyle{plain}
\newtheorem{thm}{\protect\theoremname}
  \theoremstyle{plain}
  \newtheorem{lem}[thm]{\protect\lemmaname}

\makeatother

\usepackage{babel}
  \providecommand{\lemmaname}{Lemma}
\providecommand{\theoremname}{Theorem}

\newtheorem{theorem}{Theorem}

\newtheorem{corollary}[theorem]{Corollary}

\newtheorem{definition}[theorem]{Definition}

\ifx\pdfoutput\relax\let\pdfoutput=\undefined\fi
\newcount\msipdfoutput
\ifx\pdfoutput\undefined\else
\ifcase\pdfoutput\else
\ifx\paperwidth\undefined\else
\ifdim\paperheight=0pt\relax\else\pdfpageheight\paperheight\fi
\ifdim\paperwidth=0pt\relax\else\pdfpagewidth\paperwidth\fi
\fi\fi\fi
\begin{document}

\title[A surjection theorem]{A surjection theorem for maps with singular perturbation and loss of derivatives}

\author[I. Ekeland]{Ivar Ekeland}
\address{CEREMADE, Universit\'e Paris-Dauphine, PSL Research University, CNRS, UMR 7534, Place de Lattre de Tassigny, F-75016 Paris, France} 
\email{ekeland@ceremade.dauphine.fr}

\author[\'E. S\'er\'e]{\'Eric S\'er\'e}
\address{CEREMADE, Universit\'e Paris-Dauphine, PSL Research University, CNRS, UMR 7534, Place de Lattre de Tassigny, F-75016 Paris, France} 
\email{sere@ceremade.dauphine.fr}

\date{May 27, 2021}

\begin{abstract}
In this paper we introduce a new algorithm for solving perturbed nonlinear functional equations which admit
a right-invertible linearization, but with an inverse that loses derivatives and may blow up when the perturbation parameter $\varepsilon$ goes to zero. These equations are of the form $F_\varepsilon(u)=v$ with $F_\varepsilon(0)=0$, $v$ small and given, $u$ small and unknown. The main difference with the by now
classical Nash-Moser algorithm is that, instead of using a regularized Newton scheme, we solve a sequence of Galerkin problems thanks to a topological argument. As a consequence, in our estimates there are \textit{no quadratic terms}.  For problems without perturbation parameter, our results require weaker regularity assumptions on $F$ and $v$ than earlier ones, such as those of H\"{o}rmander \cite{Hormander}. For singularly perturbed functionals $F_\varepsilon$, we allow $v$ to be larger than in previous works. To illustrate this, we apply our method to a nonlinear Schr\"{o}dinger Cauchy problem with concentrated initial data studied by Texier-Zumbrun \cite{TZ}, and we show that our result improves significantly on theirs.\end{abstract}
\maketitle

\bigskip

\bigskip

\section{Introduction}

The basic idea of the inverse function theorem (henceforth IFT) is that, if a
map $F$ is differentiable at a point $u_{0}$ and the derivative $DF\left(
u_{0}\right)  $ is invertible, then the map itself is invertible in some
neighbourhood of $u_{0}$. It has a long and distinguished history (see
\cite{BB} for instance), going back to the inversion of power series in the
seventeenth century, and has been extended since to maps between
infinite-dimensional spaces. If the underlying space is Banach, and if one is only interested
in the local surjectivity of $F$, that is, the existence, near $u_0$, of a solution $u$ to the equation $F(u)=v$ for $v$ close to $F(u_0)$,
one just needs to assume that $F$ is of class $C^1$ and that $DF(u_0)$
has a right-inverse $L(u_0)$. The standard proof is based
on the Picard scheme:
$$u_{n}=u_{n-1}-L(u_0)(F(u_{n-1})-v)$$
which converges geometrically to a solution of $F(u)=v$ provided $\Vert F(u_0)-v\Vert$
is small enough. In the $C^2$ case, the Newton algorithm:
$$u_{n}=u_{n-1}-L(u_{n-1})(F(u_{n-1})-v)$$
uses the right-invertibility of $DF(u)$ for $u$ close to $u_0$, and provides local quadratic convergence.
\medskip

In functional analysis, $u$ will typically be a function. In many situations the IFT on Banach spaces will be enough, but in the study of Hamiltonian systems
and PDEs, one encounters cases when the right-inverse $L(u)$ of $DF(u)$
loses derivatives, i.e. when $L(u) F(w)$ has less derivatives than $u$ and $w$. In such a case, the Picard and Newton schemes lose derivatives at each step.
The first solutions to this problem are due, on the one hand, to Kolmogorov \cite{Kolmogorov} and Arnol'd \cite{Arnold},
\cite{Arnold2}, \cite{Arnold3}
who investigated perturbations of completely integrable Hamiltonian systems
in the analytic class,
and showed that invariant tori persist under small perturbations, and, on the other hand, to Nash
\cite{Nash}, who showed that any smooth compact Riemannian manifold can be imbedded
isometrically into an Euclidian space of sufficiently high dimension\footnote{Nash's\ theorem on isometric embeddings was later re-proved by Gunther
\cite{Gunther}, who found a different formulation of the problem and was able to use the classical IFT in Banach spaces.}.

In both cases, the fast convergence of Newton's scheme was used to overcome the loss of regularity.
Since Nash was considering functions with finitely many derivatives, he had to introduce a sequence of smoothing operators $\mathcal{S}_n$, in order to regularize $L(u_{n-1})(F(u_{n-1})-v)$, and the new scheme was
$$u_n=u_{n-1}-\mathcal{S}_n L(u_{n-1})(F(u_{n-1})-v)\,.$$

An early presentation of Nash's method can be found in Schwartz' notes \cite{Schwartz}. It was further improved by Moser \cite{Moser}, who used it to extend the
Kolmogorov-Arnol'd results to $C^{k}$ Hamiltonians. The Nash-Moser method has been the source of a considerable amount of work in
many different situations, giving rise in each case to a so-called "hard" IFT. We will not attempt to review this line of work in the present
paper. A survey up to 1982 will be found in \cite{Hamilton}. In \cite{Hormander}, H\"{o}rmander introduced a refined version of the Nash-Moser scheme which provides
the best estimates to date on the regularity loss. We refer to \cite{Alinhac} for a pedagogical account of this work, and to \cite{BH} for recent improvements. We also gained much insight into the Nash-Moser scheme from the papers \cite{BB1}, \cite{BB2}, \cite{BB0},
\cite{BBP}, \cite{TZ}.

The question we want to address here is the following.\ The IFT\ implies that the range of $F$ contains a neighborhood $\mathcal V$ of $v_0=F(u_0)$. What is the size of $\mathcal V$? 

In general, when one tries to apply directly the abstract Nash-Moser theorem, the estimates which can be derived from its proof are unreasonably
small, many orders of magnitude away from what can be observed in numerical simulations or physical experiments.
Moreover, precise estimates for the Nash-Moser method are difficult to
compute, and most theoretical papers simply do not address the question.

So we shall address instead a ''hard'' singular perturbation problem with loss of derivatives. The same issue appears in such problems, as we shall explain in a moment, but it takes a simpler form: one tries to find a good estimate on the size of $\mathcal V$ as a power of the perturbation parameter $\varepsilon$. Such an asymptotic analysis has been carefully done in the paper of Texier and Zumbrun \cite{TZ} which has been an important source of inspiration to us, and we will be able to compare our results with theirs. As noted by these authors, the use of Newton's scheme implies an intrinsic limit to the size of $\mathcal{V}$.

Let us explain this in the ``soft'' case, without loss of derivatives. Suppose that for every $0<\varepsilon\leq 1$ we have a $C^2$ map $F_
\varepsilon$ between two Banach spaces $X$ and $Y$, such that $F_\varepsilon(0)  =0$, and, for all
$\left\Vert
u\right\Vert \leq R$,
\begin{align*}
|||\, D_{u}F_\varepsilon(u)  ^{-1}|||  &
\leq\varepsilon^{-1}M\\
|||\, D_{uu}^{2}F_\varepsilon(u) \, |||  &  \leq K
\end{align*}
Then the Newton-Kantorovich Theorem (see \cite{Ciarlet}, section 7.7 for a
comprehensive discussion) tells us that the solution $u_{\varepsilon}$ of
$F_\varepsilon(u)  =v$ exists for $\left\Vert v\right\Vert
<\frac{\varepsilon^{2}}{2KM^{2}}$, and this is essentially the best result one can hope for
using Newton's algorithm, as mentioned by Texier and Zumbrun in \cite{TZ}, Remark 2.22.
Note that the use of a Picard iteration would give a similar condition.

However, in this simple situation where no derivatives are lost, it is possible, using topological arguments instead of Newton's method,
to find a solution $u$ provided
$\left\Vert v\right\Vert \leq\varepsilon R/M$: one order of magnitude in $\varepsilon$ has
been gained. The first result of this kind, when $F$ is $C^1$ and dim$\,X\,=\,$dim$\,Y\,<\,\infty\,$, is due to Wazewski \cite{W} who used a continuation method. See also \cite{John} and \cite{Soto} and the references in these papers, for more general results in this direction.
In \cite{IE3} (Theorem 2), using Ekeland's variational principle, Wazewski's result is proved in Banach spaces, assuming only that $F$ is continuous and G\^ateaux differentiable,
the differential having a uniformly bounded right-inverse (in \S 2 below, we recall this result, as Theorem \ref{thm1}).

Our goal is to extend such a topological approach to ``hard'' problems with loss of derivatives,
which up to now have been tackled by the Nash-Moser algorithm. A first
attempt in this direction was made in \cite{IE3} (Theorem 1), in the case when the estimates on the
right-inverse do not depend on the base point,
but it is very hard to find examples of such situations. The present paper fulfills the program in the general case, where estimates on
the inverse depend on the base point.

In \cite{BBP}, Berti, Bolle and Procesi prove a new version of the Nash-Moser theorem by solving a sequence of Galerkin problems $\Pi'_nF(u_n)=\Pi'_n v$, $u_n\in E_n$, where $\Pi_n$ and $\Pi'_n$ are projectors and $E_n$ is the range of $\Pi_n$. They find the solution of each projected equation thanks to a Picard iteration:
$$u_n=\lim_{k\to\infty} w^k \ \hbox{ with } \ w^0=u_{n-1} \ \hbox{ and }
\ w^{k+1}=w^k-L_{n}(u_{n-1})(F(w^k)-v)\,,$$
where $L_n(u_{n-1})$ is a right inverse of $D(\Pi'_nF_{\,\vert_{E_n}})(u_{n-1})$. So, in \cite{BBP} the regularized Newton step is not really absent: it is essentially the first step in each Picard iteration. As a consequence, the proof in \cite{BBP} involves quadratic estimates similar to the ones of more standard Nash-Moser schemes.
Moreover, Berti, Bolle and Procesi assume the right-invertibility of $D(\Pi'_nF_{\,\vert_{E_n}})(u_{n-1})$. This assumption is perfectly suitable for the applications they consider (periodic solutions of a nonlinear wave equation), but in general it is not a consequence of the right-invertibility of $DF(u_{n-1})$, and this restricts the generality of their method as compared with the standard Nash-Moser scheme.

As in \cite{BBP}, we work with projectors and solve a sequence of Galerkin problems. But in contrast with \cite{BBP}, the Newton steps are completely absent in our new algorithm, they are replaced by the topological
argument from  \cite{IE3} (Theorem 2), ensuring the solvability of each projected equation. Incidentally, this allows us to work with functionals $F$ that are only continuous and G\^ateaux-differentiable, while the standard Nash-Moser scheme requires twice-differentiable functionals. Our regularity assumption on $v$ also seems to be optimal, and even weaker than in \cite{Hormander}. Moreover, our method works assuming either the right-invertibility of $D(\Pi'_nF_{\,\vert_{E_n}})(u)$ as in \cite{BBP}, or the right-invertibility of $DF(u)$ (in the second case, our proof is more complicated). But in our opinion, the main advantage of our approach is the following: there are \textit{no more quadratic terms} in our estimates, as a consequence we can deal with larger $v$'s, and this advantage is particularly obvious in the case of singular perturbations.

To illustrate this, we will give an abstract existence theorem with a precise estimate of the range of $F$ for
a singular perturbation problem: this is Theorem 3 below. Comparing our result with the abstract theorem of \cite{TZ}, one can see that we have weaker assumptions and a stronger conclusion.  Then we will apply Theorem 3 to an example given in \cite{TZ}, namely a Cauchy problem for a quasilinear Schr\"{o}dinger system first studied by M\'{e}tivier
and Rauch \cite{MR}. Texier and Zumbrun use their abstract Nash-Moser theorem to prove the existence of solutions of this system on a fixed time interval, for concentrated initial data.
Our abstract theorem allows us to increase the order of magnitude of the oscillation in the initial data. After reading our paper, Baldi and Haus \cite{BHperso} have been able to increase even more this order of magnitude, using their own version \cite{BH} of the Newton scheme for Nash-Moser, combined with a clever modification of the norms considered in \cite{TZ} and an improved estimate on the second derivative of the functional. In contrast, our proof follows directly from our abstract theorem, taking exactly the same norms and estimates as in \cite{TZ}, and without even considering the second derivative of the functional.

The paper is constructed as follows. In Section 2, we present the
general framework: we are trying to solve the equation $F_\varepsilon (u)
=v$ near $F_\varepsilon (0) = 0 $, when $F_\varepsilon$ maps a scale of Banach spaces of functions into another and
admits a right-invertible G\^ateaux differential
with ``tame estimates" involving losses of
derivatives and negative powers of $\varepsilon$. After giving our precise assumptions, we state
our main theorem. Section 3 is devoted to its proof. In Section 4,
we apply it to the example taken from Texier and Zumbrun \cite{TZ}, and we compare our results with theirs.\medskip

{\bf Acknowledgement.} We are grateful to Massimiliano Berti, Philippe Bolle, Jacques Fejoz and Louis Nirenberg for their interest in our work and their encouragements. It is a pleasure to thank Pietro Baldi for stimulating discussions in Naples and Paris, for a careful reading of the present paper and for a number of suggestions. We also thank the referees, whose remarks have helped us to improve this manuscript.

\bigskip

\bigskip

\section{Main assumptions and results.}

\subsection{Two tame scales of Banach spaces}

Let $(V_{s},\,\Vert\cdot\Vert_{s})_{0\leq s \leq S}$ be a scale of
Banach spaces, namely:
\[
0\leq s_{1}\leq s_{2}\leq S\Longrightarrow\left[  V_{s_{2}}\subset
V_{s_{1}}\text{ \ and\ \ }\Vert\cdot\Vert_{s_{1}}\leq\Vert\cdot\Vert_{s_{2}%
}\right]\;.
\]

We shall assume that to each $\Lambda\in [1,\infty)$ is associated a continuous linear projection $\Pi(\Lambda)$ on
$V_0$, with range $E(\Lambda)\subset V_S$. We shall also assume that the spaces $E(\Lambda)$ form a nondecreasing family of sets indexed by $[1,\infty)$, while the spaces $Ker\,\Pi(\Lambda)$ form a nonincreasing family. In other words:

\[1\leq \Lambda\leq \Lambda'\,\Longrightarrow \,\Pi(\Lambda)\Pi(\Lambda')=\Pi(\Lambda')\Pi(\Lambda)=\Pi(\Lambda)\;.\]
Finally, we assume that the projections $\Pi(\Lambda)$ are ``smoothing operators" satisfying the following estimates:
\medskip

\textbf{Polynomial growth and approximation}: \textit{There are constants }$A_{1},\ A_{2}\geq
1$\textit{ such that, for all numbers} $0\leq s\leq S$\textit{, all }$\Lambda\in [1,\infty)$\textit{ and all }$u\in V_{s}\,$\textit{, we have:}%
\begin{align}
\forall t\in [0,S]\,,\;\;\Vert\Pi(\Lambda)u\Vert_{t}  &  \leq A_{1}\,\Lambda^{(t-s)^{+}}\Vert u\Vert_{s}%
\label{loss}\\
\forall t\in [0,s]\,,\;\;\Vert(1-\Pi({\Lambda}))u\Vert_{t}  &  \leq A_{2}\,\Lambda^{-(s-t)}\Vert u\Vert_{s}
\label{gain}%
\end{align}

When the above properties are met, we shall say that $(V_{s}\,,\,\Vert\cdot\Vert_{s})_{0\leq s \leq S}$
endowed with the family of projectors $\left\{\,\Pi(\Lambda)\;,\;\Lambda\in [1,\infty)\,\right\}\,, $ is a {\it tame} Banach scale.
\medskip

It is well-known (see e.g. \cite{BBP}) that (\ref{loss},\ref{gain}) imply: \medskip

\textbf{Interpolation inequality}:\textit{ For }$0\leq t_{1}%
\,\leq\,s\,\leq\,t_{2}\leq S$\textit{ ,}%
\begin{equation}
\Vert u\Vert_{s}\leq A_{3}\Vert u\Vert_{t_{1}}^{\frac{t_{2}-s}{t_{2}-t_{1}}%
}\Vert u\Vert_{t_{2}}^{\frac{s-t_{1}}{t_{2}-t_{1}}}\;. \label{inter}%
\end{equation}

Let $(W_{s}\,,\,\Vert\cdot\Vert'_s)_{0\leq s\leq S}$ be another tame scale of Banach spaces. We shall denote by $\Pi^{\prime}(\Lambda)$ the corresponding projections defined on $W_0$ with ranges $E'(\Lambda)\subset W_S$, and by $A'_{i}\;(i=1,2,3)$ the
corresponding constants in (\ref{loss}), (\ref{gain}) and (\ref{inter}).\medskip

{\bf Remark.} In many practical situations, the projectors form a discrete family as, for instance, $\{\Pi(N)\,,\;N\in\mathbb{N}^*\}$, or $\{\Pi({2^j})\,,\;j\in\mathbb{N}\}$. The first case occurs when $\Pi(N)$ acts
on periodic functions by truncating their Fourier series, keeping only frequencies of size less or equal to $N$,
as in \cite{BBP}.
The second case occurs when truncating orthogonal wavelet expansions as in an earlier version of the present work \cite{ESDebut}. Our choice of notation and assumptions covers these
cases, taking $\Pi(\Lambda)=\Pi(\left\lfloor \Lambda\right\rfloor)$ or $\Pi(\Lambda)=\Pi(2^{\left\lfloor \log_2(\Lambda)\right\rfloor})$, where $\left\lfloor \cdot\right\rfloor$ denotes the integer part.\medskip

\subsection{Main theorem}

We state our result in the framework of singular perturbations, in the spirit
of Texier and Zumbrun \cite{TZ}. The norms $\Vert\cdot\Vert_s\,,\,\Vert\cdot\Vert_s'$ on the tame scales $(V_{s})$, $(W_{s})$ may depend on the perturbation parameter $\varepsilon
\in(0,1]$, as well as the projectors $\Pi(\Lambda)\,,\,\Pi'(\Lambda)\,$ and their ranges $E(\Lambda)$, $E'(\Lambda)\,.$ But we impose that
$S$ and the constants $A_i,\,A'_i$ appearing in estimates (\ref{loss},
\ref{gain}, \ref{inter}) be independent of $\varepsilon$. In order to avoid
burdensome notations, the dependence of the norms, projectors and subspaces on
$\varepsilon$ will not be explicit in the sequel.\bigskip

Denote by $B_{s}$ the unit ball in
$V_{s}$:%
\[
B_{s}=\left\{  u\ |\ \left\Vert u\right\Vert _{s}\leq1\right\}
\]

In the sequel we fix nonnegative constants $s_0, m,\ell,\,\ell'$ and $g$, independent of $\varepsilon$. We will assume that $S$ is large enough.\bigskip

We first recall the definition of G\^ateaux-differentiability, in a form adapted to our framework:\medskip

\begin{definition}
We shall say that a function $F:\,B_{s_{0}+m}\rightarrow
W_{s_{0}}$ is \emph{G\^{a}teaux-differentiable} (henceforth
G-differentiable) if for every $u\in
B_{s_{0}+m}$, there exists a linear map $DF\left(
u\right)  :V_{s_0+m}\rightarrow W_{s_0}$ such that for every
$s\in [s_0,S-m]$, if $u\in
B_{s_{0}+m}\cap V_{s+m}$,  then $DF \left(
u\right)$ maps continuously $V_{s+m}$ into $W_s$, and %
\[
\forall h\in V_{s+m}\ ,\ \lim_{t\rightarrow0}\ \left\Vert \frac{1}{t}\left[
F\left(  u+th\right)  -F\left(  u\right)  \right]  -D F\left(  u\right)
h\right\Vert _{s}^{\prime}=0\;.
\]

\end{definition}

\medskip
Note that, even in finite dimension, a G-differentiable map need not be $C^{1}$, or
even continuous. However, if $D F:\,V_{s+m}\to\mathcal{L}(V_{s+m},W_s)$ is locally bounded, then $F:\,V_{s+m}\to W_s$ is locally Lipschitz, hence continuous. In the present paper, we will always be in such a situation.\bigskip

We now consider a family of maps $(F_\varepsilon)_{0<\varepsilon\leq 1}$ with $F_\varepsilon:\ B_{s_0+m} \to W_{s_0}$. We are ready to state our assumptions on this family:\bigskip

\begin{definition}
\leavevmode\par
\begin{itemize}

\item We shall say that the maps $F_\varepsilon:\,B_{s_{0}+m}\rightarrow
W_{s_{0}}\; (0<\varepsilon \leq 1)$ form an $S$-tame differentiable family if they are G-differentiable with respect to
$u$, and, for some positive constant $a\,,$ for all $\varepsilon\in (0, 1]$ and all $s \in [s_{0}, S-m]\,$, if $u\in B_{s_{0}+m}\cap V_{s+m}$ and $h\in V_{s+m}\,,$ then
$DF_{\varepsilon}\left(  u\right)  h \in W_s$ with the tame direct estimate
\begin{equation}
\left\Vert DF_{\varepsilon}\left(  u\right)  h\right\Vert _{s}^{\prime}\leq
a\left(  \left\Vert h\right\Vert _{s+m}+\left\Vert u\right\Vert _{s+m}
\left\Vert h\right\Vert _{s_{0}+m}\right)\;.  \label{tamedirectepsilon}
\end{equation}

\item Then we shall say that $(DF_\varepsilon)_{0<\varepsilon\leq 1}$ is tame right-invertible if there are $\,b>0$ and $g,\,\ell,\,\ell'\geq 0$ such that for all
$0<\varepsilon\leq 1$ and $u\in B_{s_{0}+\max\{m,\ell\}}$ , there is a linear map
$L_{\varepsilon}\left(  u\right)  :W_{s_0+\ell^{\prime}}\rightarrow V_{s_0}$
satisfying
\begin{equation}
\forall k\in W_{s_0+\ell'}\,,\ \ \ DF_{\varepsilon}\left(  u\right)
L_{\varepsilon}\left(  u\right)  k=k \label{rightinverse}%
\end{equation}
and for all $s_{0}\leq s\leq
S-\max\left\{  \ell,\ell^{\prime}\right\}  $, if $u\in B_{s_{0}+\max\{m,\ell\}}\cap V_{s+\ell}\,$ and $k\in W_{s+\ell^{\prime}}\,,$ then $L_{\varepsilon}\left(
u\right)  k\in V_s\,,$ with the tame inverse estimate
\begin{equation}
\left\Vert L_{\varepsilon}\left(
u\right)  k\right\Vert _{s}\leq b\varepsilon^{-g}\left(  \left\Vert
k\right\Vert _{s+\ell^{\prime}}^{\prime}+\left\Vert k\right\Vert _{s_{0}%
+\ell^{\prime}}^{\prime}\left\Vert u\right\Vert _{s+\ell}\right)\;.
\label{tameinverseepsilon}%
\end{equation}

\item Alternatively, we shall say that $(DF_\varepsilon)_{0<\varepsilon\leq 1}$ is tame {\it Galerkin} right-invertible if there are $\,\underline{\Lambda}\geq 1\,$, $b>0$ and $\,g,\,\ell,\,\ell'\geq 0$ such that for all $\Lambda\geq \underline{\Lambda}\,,\ 0<\varepsilon\leq 1$ and any $u\in B_{s_{0}+\max\{m,\ell\}}\cap E(\Lambda)$, there is a linear map
$L_{\Lambda,\varepsilon}\left(  u\right)  :E'(\Lambda)\rightarrow E(\Lambda)$
satisfying%
\begin{equation}
\forall k\in E'(\Lambda)\,,\ \ \ \Pi'(\Lambda) DF_{\varepsilon}\left(  u\right)
L_{\Lambda,\varepsilon}\left(  u\right)  k=k \label{galrightinverse}%
\end{equation}
and for all $s_{0}\leq s\leq
S-\max\left\{  \ell,\ell^{\prime}\right\}  $, we have the tame inverse estimate:
\begin{equation}
\forall k\in E'(\Lambda)\,, \ \left\Vert L_{\Lambda,\varepsilon}\left(
u\right)  k\right\Vert _{s}\leq b\varepsilon^{-g}\left(  \left\Vert
k\right\Vert _{s+\ell^{\prime}}^{\prime}+\left\Vert k\right\Vert _{s_{0}%
+\ell^{\prime}}^{\prime}\left\Vert u\right\Vert _{s+\ell}\right)\;.
\label{galtamegalinverseepsilon}%
\end{equation}

\end{itemize}
\end{definition}

In this definition, the integers $m, \ell, \ell^{\prime}$ denote the loss of derivatives for $DF_{\varepsilon}$ and its right-inverse, and $g>0 $ denotes the strength of the singularity at $\varepsilon = 0$. The unperturbed case of a fixed map can be recovered by setting $\varepsilon =1$.

We want to solve the equation 
$F_\varepsilon(u)=v$.
There are three things to look for. How regular is $v$ ? How regular is $u$, or, equivalently, how small is the loss of derivatives between $v$ and $u$ ? How does the existence domain depend on $\varepsilon$ ? The following result answers them in a near-optimal way.

\begin{theorem}
\label{Thm8}

Assume that the maps $F_\varepsilon$ $(0<\varepsilon\leq 1)$ form an $S$-tame differentiable family between the tame scales $(V_s)_{0\leq s\leq S}$ and $(W_s)_{0\leq s\leq S}$, with $F_\varepsilon(0)=0$ for all $0<\varepsilon\leq 1$. Assume, in addition, that $(DF_\varepsilon)_{0<\varepsilon\leq 1}$ is either tame right-invertible or tame Galerkin right-invertible. Let $s_0,\,m,\,g,\,\ell,\,\ell'$ be the associated parameters.

Let $s_1 \geq s_0+\max\{m,\ell\}$, $\delta>s_1+\ell'$ and $g'>g$.

Then, for $S$ large enough, there is $r>0$ such that, whenever $0<\varepsilon\leq 1$ and $\left\Vert v\right\Vert _{\delta}^{\prime}\leq r\varepsilon^{g^{\prime}}$,
there exists some $u_{\varepsilon}\in B_{s_1}$ satisfying:
\begin{align*}
&F_\varepsilon(u_{\varepsilon})=v\ \\
&\left\Vert u_{\varepsilon}\right\Vert _{s_1}\leq r^{-1}\,\varepsilon^{-g^{\prime}
}\left\Vert v\right\Vert _{\delta}^{\prime}
\end{align*}
\end{theorem}
\bigskip

As we will see, the proof of Theorem \ref{Thm8} is much shorter under the assumptions that $DF_\varepsilon$ is Galerkin right-invertible. But in many applications, it is easier to check that $DF_\varepsilon$ is tame right-invertible
than tame Galerkin right-invertible. See \cite{BBP}, however, where an assumption similar to (\ref{galrightinverse}, \ref{galtamegalinverseepsilon}) is used.

All other ``hard" surjection theorems that we know of require some additional conditions on the second derivative of  $F_\varepsilon$. Here we do not need such assumptions, in fact we only assume $F_\varepsilon$ to be G-differentiable, not $C^2$. 

As for the three questions we raised, let us explain in which sense the answers are almost optimal in Theorem \ref{Thm8}. For the tame estimates \eqref{tamedirectepsilon},\eqref{tameinverseepsilon} to hold, one needs $u\in B_{s_1}$ with $s_1\geq s_0 +\max\{m,\ell\}$. When solving the linearized equation $DF_{\varepsilon}\left(  u\right) h=k$ in $V_{s_1}$ by $h =L_{\varepsilon}\left(  u\right) k$, one needs $k\in W_{s_1+\ell'}\,$, so it seems necessary to assume
$\delta\geq s_1+\ell'\,,$
and we find that the strict inequality is sufficient. Replacing $s_1$ with its minimal value, our condition on $\delta$ becomes 
\begin{equation*}
    \delta > s_0 +\max\{m,\ell\}+\ell'\,.
\end{equation*}
We have not found this condition in the literature: in \cite{Hormander} for instance, a stronger assumption is made, namely $\delta >s_0+\max\{2m+\ell',\ell\}+\ell'$.

For the dependence of $\Vert v\Vert'_\delta$ on $\varepsilon$, the constraint $g'>g$ also seems to be nearly optimal. Indeed, the solution $u_\varepsilon$ has to be in $B_{s_1}$, but the right-inverse $L_\varepsilon$ of $DF_{\varepsilon}$ has a norm of order $\varepsilon^{-g}$, so the condition $\Vert v\Vert'_\delta\lesssim \varepsilon^g$ seems necessary. We find that the condition $\Vert v\Vert'_\delta\lesssim \varepsilon^{g'}$ is sufficient.

Our condition on $S$ is of the form $S\geq S_0$ where $S_0$ depends only on the parameters $s_0,\,m,\,g,\,\ell,\,\ell'$ and $g',\,s_1\,,\delta$. Then $r$ depends only on these parameters and the constants $A_i$, $A'_i$ associated with the tame scales. In principle, all these constants could be made explicit, but we will not do it here. Let us just mention that one can take $S_0=\mathcal{O}\left(\frac{1}{g'-g}\right)$ as $g'\to g$, all other parameters remaining fixed. This follows from the inequality $\sigma<\zeta g/\eta$ in Lemma \ref{compatibility}.

\medskip

In the case of a tame right-invertible differential, we can restate our theorem in a form that allows direct comparison with \cite{TZ}: {\it Theorem 2.19 and Remarks 2.9, 2.14}. For this purpose, we consider two tame Banach scales $(V_s,\Vert\cdot\Vert_s)$ and $(W_s,\Vert\cdot\Vert'_s)$ with associated projectors $\Pi_\Lambda\,,\,\Pi'_\Lambda$, we take
$\gamma>0$ and we introduce the norms $\vert \cdot\vert_{s}:=\varepsilon^\gamma\Vert \cdot\Vert_{s}$ and $\vert \cdot\vert'_{s}:=\varepsilon^\gamma\Vert \cdot\Vert'_{s}\,$.
We then denote ${\mathfrak B}_{s}(\rho)=\left\{  u\ |\ \left\vert u\right\vert _{s}\leq \rho\right\}$ and we consider
functions $F_{\varepsilon}$ of the form $F_\varepsilon(u)=\Phi_\varepsilon({\mathfrak a}_\varepsilon+u)-\Phi_\varepsilon({\mathfrak a}_\varepsilon)\,,$ where $\Phi_\varepsilon$ is defined on ${\mathfrak B}_{s_0+m}(2\varepsilon^\gamma)\,$ and ${\mathfrak a}_\varepsilon\in {\mathfrak B}_{S}(\varepsilon^\gamma)\,$ is chosen such that $v_\varepsilon:=-\Phi_\varepsilon({\mathfrak a}_\varepsilon)$ is very small. A point $u$ in $B_{s_0+m}$ satisfies $F_\varepsilon(u)=v_\varepsilon\,$ if and only if it solves the equation $\Phi_\varepsilon({\mathfrak a}_\varepsilon+u)=0$ in ${\mathfrak B}_{s_0+m}(\varepsilon^\gamma)\,.$
We make the following assumptions on $\Phi_\varepsilon$:
\newpage

{\it For some $\gamma>0$ and any $0<\varepsilon \leq 1$, the map $\,\Phi_{\varepsilon}\,:\; {\mathfrak B}_{s_{0}+m}(2\varepsilon^\gamma)\rightarrow
W_{s_{0}}$ is G-differentiable with respect to
$u$, and there are constants $a$, $b$ and $g>0$ such that:$\ $

\begin{itemize}
\item for all $0<\varepsilon\leq 1$ and $\ s_{0}\leq s\leq S-m\,,$ if $u\in {\mathfrak B}_{s_{0}+m}(2\varepsilon^\gamma)\cap V_{s+m}$ and $h\in V_{s+m}\,,$
then $D\Phi_{\varepsilon}\left(  u\right)  h\in W_s\,,$ with the tame direct estimate
\begin{equation}
\left\vert D\Phi_{\varepsilon}\left(  u\right)  h\right\vert _{s}^{\prime}\leq
a\left(  \left\vert h\right\vert _{s+m}+\varepsilon^{-\gamma}\left\vert u\right\vert _{s+m}%
\left\vert h\right\vert _{s_{0}+m}\right)  \label{tamedirectepsilonbis}%
\end{equation}

\item for all $0<\varepsilon\leq 1$ and $ u\in {\mathfrak B}_{s_{0}+\max\{m,\ell\}}(2\varepsilon^\gamma)\,$, there is
$L_{\varepsilon}\left(  u\right)  :\,W_{s_0+\ell^{\prime}}\rightarrow V_{s_0}$ linear,
satisfying:%
\begin{equation}
\forall k\in W_{s_0+\ell^{\prime}},\ \ \ D\Phi_{\varepsilon}\left(  u\right)
L_{\varepsilon}\left(  u\right)  k=k \label{rightinversebis}%
\end{equation}
and for all $s_{0}\leq s\leq
S-\max\left\{  \ell,\ell^{\prime}\right\}\,,$ if $ u\in {\mathfrak B}_{s_{0}+\max\{m,\ell\}}(2\varepsilon^\gamma)\cap V_{s+\ell}\,$ and $k\in W_{s+\ell^{\prime}}\,$, then
$L_{\varepsilon}\left(
u\right)  k\in V_s\,,$ with the tame inverse estimate
\begin{equation}
\left\vert L_{\varepsilon}\left(
u\right)  k\right\vert _{s}\leq b\varepsilon^{-g}\left(  \left\vert
k\right\vert _{s+\ell^{\prime}}^{\prime}+\varepsilon^{-\gamma}\left\vert k\right\vert _{s_{0}%
+\ell^{\prime}}^{\prime}\left\vert u\right\vert _{s+\ell}\right)
\label{tameinverseepsilonbis}%
\end{equation}

\end{itemize}
}
\bigskip

Under these assumptions, the maps $F_\varepsilon:\,B_{s_0+m}\to W_{s_0}$ form an $S$-tame differentiable family for the ``old"  norms $\Vert\cdot\Vert_s\,$, $\Vert\cdot\Vert'_s\,$. So the following result holds, as a direct consequence of our main theorem:\bigskip

\begin{corollary}
\label{Cor9} Consider two tame Banach scales $(V_s,\vert\cdot\vert_{s})_{0\leq s\leq S}$ and $(W_s,\vert\cdot\vert'_{s})_{0\leq s\leq S}\,$, nonnegative constants $s_0,\, m,\ell,\,\ell',\,g,\,\gamma$, and two positive constants $a,\,b$. Take any $g^{\prime}>g$, $s_1\geq s_0+\max\{m,\ell\}$ and $\delta>s_1+\ell'$. For $S$ large enough and $r>0$ small, if a family of G-differentiable maps $\Phi_\varepsilon\,:\; {\mathfrak B}_{s_{0}+m}(2\varepsilon^\gamma)\rightarrow
W_{s_{0}}$ $(0< \varepsilon \leq 1)$ satisfies (\ref{tamedirectepsilonbis},\ref{rightinversebis},\ref{tameinverseepsilonbis}), and, in addition, for
some ${\mathfrak a}_\varepsilon\in {\mathfrak B}_{S}(\varepsilon^\gamma)\,,\,$
$\vert \Phi_\varepsilon({\mathfrak a}_\varepsilon)\vert _{\delta}^{\prime}\leq r\varepsilon^{\gamma+g^{\prime}}$, then
there exists some $u_{\varepsilon}\in {\mathfrak B}_{s_1,\varepsilon}(\varepsilon^\gamma)$ such that:
\begin{align*}
&\Phi_\varepsilon({\mathfrak a}_\varepsilon+u_{\varepsilon})  =0\ \\
&\vert u_{\varepsilon}\vert _{s_1} \leq r^{-1}\, \varepsilon^{-g^{\prime}}
\vert \Phi_\varepsilon({\mathfrak a}_\varepsilon)\vert _{\delta}^{\prime}
\end{align*}
\end{corollary}
\bigskip

In \cite{TZ} ({\it Theorem 2.19 and Remarks 2.9, 2.14}), the assumptions are stronger, since they involve the second derivative of $\Phi_\varepsilon$. More importantly, we only need the norm of $\Phi_\varepsilon({\mathfrak a}_\varepsilon)$ to be controlled by $\varepsilon^{\gamma+g'}$ with $g'>g$, provided $S\geq S_0$ with $S_0=\mathcal{O}\left(\frac{1}{g'-g}\right)$, while in \cite{TZ} (Assumption 2.15 and Remark 2.23), due to quadratic estimates, one needs $g'>2g$ with the faster growth $S_0=\mathcal{O}\left(\frac{1}{(g'-2g)^2}\right)$.\bigskip

\section{Proof of Theorem \ref{Thm8}}
\bigskip

The proof consists in constructing a sequence $(u_{n})_{n\geq 1}$ which
converges to a solution $u$ of $F\left(u\right)=v$. At each step, in order to find $u_n$, we solve a nonlinear equation in a Banach space, using Theorem 2 in \cite{IE3}, which we restate below for the reader's convenience (the notation $||| \, L\, |||$ stands for the operator norm of any linear continuous map $L$ between two Banach spaces):\bigskip

\begin{theorem}
\label{thm1} Let $X$ and $Y$ be Banach spaces. Let $f:B_X(0,R)\rightarrow Y$ be
continuous and G\^{a}teaux-differentiable, with $f\left(  0\right)  =0$.
Assume that the derivative $Df\left(  u\right)  $ has a right-inverse
$L\left(  u\right)  $, uniformly bounded on the ball $B_X(0,R)$:
\begin{align*}
&\forall (u,k)\in B_X(0,R)\times Y\text{, \ }Df\left(  u\right) L\left(  u\right)   \,  k=k\ \\
&\sup\left\{ \, |||\, L\left(  u\right)  ||| \ :\ \left\Vert
u\right\Vert_X < R\right\}   <M\,.
\end{align*}

Then, for every $v\in Y$ with $\left\Vert v\right\Vert_Y < RM^{-1}$ there is
some $u\in X$ satisfying:
\[
f\left(  u\right)  =v\;\text{ and }\;\left\Vert u\right\Vert_X \leq M\left\Vert v\right\Vert_Y < R\,.
\]

\end{theorem}
\bigskip

Note first that this is a local surjection theorem, not an inverse function
theorem:\ with respect to the IFT, we lose uniqueness.\ On the other hand, the
regularity requirement on $f$ and the smallness condition on $v$ are much weaker. As mentioned in the Introduction, for a $C^1$ functional in finite dimensions, this theorem has been proved a long time ago by Wazewski \cite{W}
by a continuation argument (we thank Sotomayor for drawing our attention to this result).
For a comparison of the existence and uniqueness domains in the $C^2$ case with dim$\,X\,=\,$dim$\,Y\,$, see \cite{Hart}, chapter II, exercise 2.3.\bigskip

It turns out that the proof of Theorem \ref{Thm8} is much easier if one assumes that the family $(DF_\varepsilon)$ is tame {\it Galerkin} right-invertible. But most applications require that $(DF_\varepsilon)$ be tame right-invertible. Let us explain why the proof is longer in this case.
In our algorithm, we will use two sequences of projectors $\Pi_n:=\Pi(\Lambda_n)$ and
$\Pi'_n:=\Pi'(M_n)$ with associated ranges $E_n=E(\Lambda_n)$ and $E'_n=E'(M_n)$, where $\Lambda^0\approx \varepsilon^{-\eta}$ for some small $\eta>0$, 
$\Lambda_n=\Lambda_0^{\alpha^n}$ for some $\alpha>1$ close to $1$, and $M_n=\Lambda_n^{\vartheta}$
for some $\vartheta\leq 1$ such that $\vartheta\alpha>1$. The algorithm consists in finding, by induction on $n$ and using Theorem \ref{thm1} at each step, a solution $u_n\in E_n$ of the problem
$\Pi'_nF_\varepsilon(u_n)=\Pi'_{n-1}v$. For this, we need $\Pi'_n DF_\varepsilon(u)_{\vert_{E_n}}$ to be invertible for $u$
in a certain ball $\mathcal{B}_n$, with estimates on the right inverse for a certain norm $\Vert\cdot\Vert_{\mathcal{N}_n}$.

When the family $(DF_\varepsilon)$ is tame Galerkin right-invertible,
we can take $\vartheta=1$ so that $M_n=\Lambda_n$, instead of assuming $\vartheta<1$. Then
the right-invertibility of $\Pi'_n DF_\varepsilon(u)_{\vert_{E_n}}$ follows immediately from the definition.

But when $(DF_\varepsilon)$ is only tame right-invertible,
it is crucial to take $\vartheta<1$. The intuitive idea is the following. One can think of $DF_\varepsilon(u)$ as very large right-invertible matrix. The topological argument we use requires $\Pi'_n DF_\varepsilon(u)_{\vert_{E_n}}$ to have a right-inverse for $u$ in a suitable ball. If we take $M_n=\Lambda_n$, this is like asking that a square submatrix of a right-invertible matrix be invertible. In general this is not true. But a rectangular submatrix, with more columns than lines, will be right-invertible if the full matrix is and if there are enough columns in the submatrix. This is why we impose $M_n<\Lambda_n$ when we do not assume the tame Galerkin right-invertibility.

In the sequel, we assume that the family $(DF_\varepsilon)$ is tame right-invertible, so we take $\vartheta<1$, and we point out the specific places where the arguments would be easier assuming, instead, that $(DF_\varepsilon)$ is tame {\it Galerkin} right-invertible.

The sequence $u_n$
depends on a number of parameters $\eta, \alpha, \beta, \vartheta$ and $\sigma$ satisfying various conditions: in the first subsection we prove that these conditions are compatible. In the next one, we construct an initial point $u_1$ depending on $\eta,\,\alpha$ and $\vartheta$. In the third one we construct, by induction, the remaining points $u_n$ which also depend on $\beta$ and $\sigma$. Finally we prove that the sequence $(u_n)$ converges to a solution $u$ of the problem, satisfying the desired estimates.\newpage

\subsection{Choosing the values of the parameters}

We are given $s_1\geq s_{0}+\max\left\{ m,\ell\right\},$ $\delta> s_1 +\ell'$ and $g'>g$. These are fixed throughout the proof.\medskip

We introduce positive parameters $\eta,\alpha,\beta,\vartheta$ and $\sigma$
satisfying the following conditions:
\begin{align}
\eta\, & <\frac{g'-g}{\max\left\{ \vartheta \ell',\ell\right\}}\label{eq:i0}\\
\frac{1}{\alpha} & <\vartheta<1\label{eq:i1}\\
\left(1-\vartheta\right)\left(\sigma-\delta\right) & >\vartheta m+\max\left\{ \ell,\vartheta \ell'\right\} +\frac{g}{\eta}\label{eq:i2}\\
\sigma & >\alpha\beta+s_1\label{eq:i4}\\
\left(1+\alpha-\vartheta\alpha\right)\left(\sigma-s_{0}\right) & >\alpha\beta+\alpha\left(m+\ell\right)+\ell'+\frac{g}{\eta}\label{eq:i5}\\
\left(1-\vartheta\right)\left(\sigma-s_{0}\right) & >m+\vartheta \ell'+\frac{g}{\alpha\eta}\label{eq:i6}\\
\delta & >s_{0}+\frac{\alpha}{\vartheta}\left(\sigma-s_{0}-\alpha\beta+\ell"\right)\label{eq:i7}\\
\left(\alpha-1\right)\beta & >\left(1-\vartheta\right)\left(\sigma-s_{0}\right)+\vartheta m+\ell"+\frac{g}{\eta}\label{eq:i8}\\
\ell" & =\max\left\{ \left(\alpha-1\right)\ell+\ell',\alpha\vartheta \ell'\right\} \label{eq:i9}
\end{align}

Note that condition (\ref{eq:i2}) implies that $\delta < \sigma $ . Note also that condition (\ref{eq:i7}) may be rewritten as
$$\beta >\frac{1}{\alpha}\left(\sigma-\delta\right)+\left(1-\frac{\vartheta}{\alpha}\right)\frac{\delta-s_{0}}{\alpha}+\frac{\ell"}{\alpha}$$

which implies the simpler inequality
\begin{equation}\beta >\frac{1}{\alpha}\left(\sigma-\delta\right)\label{eq:i3}
\end{equation}

Inequality (\ref{eq:i3}) will also be used in the proof.\medskip

 If we assume tame Galerkin right-invertibility instead of tame right-invertibility, we can replace condition (\ref{eq:i2}) by the weaker condition $\delta < \sigma$, we do not need conditions (\ref{eq:i5}), (\ref{eq:i6}) any more, and we can take $\vartheta=1$ instead of $\vartheta<1$.

\bigskip

\begin{lem}
\label{compatibility}
The set of parameters $\left(\eta,\alpha,\beta,\vartheta,\sigma\right)$ satisfying
the above conditions is non-empty. More precisely, there are some ${\alpha}>1$ and ${\zeta}>0$ depending only on $(s_0,\,m,\,\ell,\,\ell',\,s_1,\,\delta)$, such that, for $\vartheta={\alpha}^{-1/2}$ and for every $0<\eta<1$, there exist $(\beta,\sigma)$
with $\sigma<{\zeta}g/\eta$ such that the constraints (\ref{eq:i2}) to
(\ref{eq:i9}) are satisfied.
\end{lem}

\begin{proof}
Since $\delta>s_{1}+\ell'$, and $\ell"\rightarrow \ell'$ when both $\alpha$ and $\vartheta$ tend to $1$,
it is possible to choose $\vartheta$ and $\alpha=\vartheta^{-2}$
close enough to $1$ so that $\delta>s_{0}+\frac{\alpha}{\vartheta}\left(s_1-s_{0}+\ell"\right)$.
Take some $\tau$ with $0<\tau<\frac{\vartheta}{\alpha}\left(\delta-s_{0}\right)-s_1+s_{0}-\ell"\text{, }$
and set:
\begin{equation}
\beta=\frac{\sigma}{\alpha}-\frac{s_1+\tau}{\alpha}\label{eq:i10}
\end{equation}
 Then conditions (\ref{eq:i1}), (\ref{eq:i4}) and (\ref{eq:i7}) are satisfied.

The remaining inequalities are constraints on $\beta $ and $\sigma $. They can be rewritten as follows:
\begin{align}
\sigma & >\delta+\frac{1}{1-\vartheta}\left[\vartheta m+\max\left\{ \ell,\vartheta \ell'\right\} +\frac{g}{\eta}\right]\label{eq:i11}\\
\beta< & \left(\frac{1}{\alpha}+1-\vartheta\right)\sigma-m-\ell-\frac{\ell'}{\alpha}-\left(\frac{1}{\alpha}+1-\vartheta\right)s_{0}-\frac{g}{\alpha\eta}\label{eq:i13}\\
\sigma & >s_{0}+\frac{1}{1-\vartheta}\left(m+\vartheta \ell'+\frac{g}{\alpha\eta}\right)\label{eq:i14}\\
\beta & >\frac{1-\vartheta}{\alpha-1}\sigma+\frac{1}{\alpha-1}\left(\vartheta m+\ell"+\frac{g}{\eta}-\left(1-\vartheta\right)s_{0}\right)\label{eq:i15}
\end{align}

These inequalities define half-planes in the $ (\sigma, \beta)$-plane. 
Since $\alpha\vartheta>1$, the slopes in (\ref{eq:i10}),
(\ref{eq:i13}) and (\ref{eq:i15}) are ordered as follows: 
\[
0<\frac{1-\vartheta}{\alpha-1} <\frac{1}{\alpha}<\frac{1}{\alpha}+1-\vartheta<1
\]

As a consequence, for the chosen values of $\alpha,\vartheta$ and $\tau$,
the domain defined by these three conditions in the $\left(\sigma, \beta\right)$-plane
is an infinite half-line stretching to the North-East. The remaining
two, (\ref{eq:i11}) and (\ref{eq:i14}), just tell us that $\sigma$
should be large enough. So the set of solutions is of the form $\sigma>\bar{\sigma}$,
$\beta=\frac{\sigma}{\alpha}-\frac{s_1+\tau}{\alpha}$ and $\bar{\sigma}$
is clearly a piecewise affine function of $g/\eta$. We may thus choose $\sigma < \zeta g / \eta$ for some constant $\zeta\,.$
\end{proof}

{\bf Remark.}
{\it As already mentioned, if we assume that $(DF_\varepsilon)$ is tame Galerkin right-invertible, (\ref{eq:i2}) can be replaced by the condition $\delta<\sigma$, and (\ref{eq:i5}) and (\ref{eq:i6}) are not needed. The remaining conditions can be satisfied by taking $\vartheta=1$ and for a larger set of the other parameters. The corresponding variant of Lemma \ref{compatibility} has a simpler proof. We can choose $\alpha>1$ such that $\delta>s_{0}+\alpha\left(s_1-s_{0}+\ell"\right)$ and $\tau$ such that $0<\tau< \frac{1}{\alpha}\left(\delta-s_{0}\right)-s_1+s_{0}-\ell"$, and we may impose condition (\ref{eq:i10}). Then conditions (\ref{eq:i11}), (\ref{eq:i13}) and (\ref{eq:i14}) are no longer required, and the last conditions $\delta<\sigma$ and (\ref{eq:i15}) are easily satisfied by taking $\sigma$ large enough.}
\bigskip

The values $\left(\eta,\alpha,\beta,\vartheta,\sigma\right)$ are now fixed.
For the remainder of the proof we introduce an important notation.
By
\[
x\lesssim y
\]
we mean that there is some constant $C$ such that $x\leq Cy$. This constant depends on $\,A_i,\,A'_i,\,a,\,b,\,s_0,\,m,\,\ell,\,\ell',\,g,\,g',\,\,s_1,\,\delta$ and our additional parameters $\left(\eta,\alpha,\beta,\vartheta,\sigma\right)$, but NOT on $\varepsilon$, nor on the regularity index $s\in [0,S]$ or the rank $n$ in any of the
sequences which will be introduced in the sequel.
For instance, the tame inequalities become:
\begin{align*}
\left\Vert DF_{\varepsilon}\left(u\right)h\right\Vert _{s} & \lesssim\left(\left\Vert u\right\Vert _{s+m}\left\Vert h\right\Vert _{s_{0}+m}+\left\Vert h\right\Vert _{s+m}\right)\\
\left\Vert L_{\varepsilon}\left(u\right)k\right\Vert _{s} & \lesssim\varepsilon^{-g}\left(\left\Vert u\right\Vert _{s+l}\left\Vert k\right\Vert _{s_{0}+l'}+\left\Vert k\right\Vert _{s+l'}\right)
\end{align*}

In the iteration process, we will need the following result:

\begin{lem}
\label{boundF} If the maps $F_\varepsilon$ form an $S$-tame differentiable family and $F_\varepsilon\left(  0\right)  =0$, then, for $u\in B_{s_0+m}\cap V_{s+m}$ and $s_{0}\leq s\leq S-m$, we have:%
\[
\left\Vert F_\varepsilon\left(  u\right)  \right\Vert _{s}\lesssim  \left\Vert u\right\Vert _{s+m}%
\]

\end{lem}

\begin{proof}
Consider the function $\varphi\left(  t\right)  =\left\Vert F_\varepsilon \left(
tu\right)  \right\Vert _{s}$. Since $F_\varepsilon$ is G-differentiable, we have:
\begin{align*}
\varphi^{\prime}\left(  t\right)   &  =\left\langle DF_\varepsilon\left(  tu\right)
u,\frac{F_{\varepsilon}\left(  tu\right)  }{\left\Vert F_{\varepsilon}\left(
tu\right)  \right\Vert _{s}}\right\rangle _{s}\\
&  \leq a\left(  t\left\Vert u\right\Vert _{s_{0}+m}\left\Vert u\right\Vert
_{s+m}+\left\Vert u\right\Vert _{s+m}\right)
\end{align*}
and since $\varphi\left(  0\right)  =0$, we get the result.
\end{proof}

\subsection{Initialization}

\subsubsection{Defining appropriate norms.}

This subsection uses condition (\ref{eq:i1}) and the inequalities $s_1+\ell'<\delta<\sigma\,$, which, as already noted, follows from (\ref{eq:i2}). \smallskip

We are given $\left(\eta,\alpha,\vartheta,\delta,\sigma\right)$. We fix a large constant $K>1$, to be chosen later independently of $0<\varepsilon\leq 1$.\smallskip

We set
$\Lambda_{0}=(K\varepsilon^{-\eta})^{1/\alpha}$, $\Lambda_{1}:=(\Lambda_{0})^{\alpha}=K\varepsilon^{-\eta}$, 
$M_{0}:= (\Lambda_{0})^{\vartheta}=(K\varepsilon^{-\eta})^{\vartheta/\alpha}$ and $M_{1}:= (\Lambda_{1})^{\vartheta}=(K\varepsilon^{-\eta})^{\vartheta}$. We then have the inequalities $M_0<\Lambda_{0} < M_{1} <\Lambda_{1}\,.$\smallskip

Let $\,E_1:=E(\Lambda_1)\,,\Pi_1:=\Pi(\Lambda_1)\,,\;E'_1=E(M_1)\,$ and $\Pi'_i:=\Pi'(M_i)\,$ for $i=0,\,1\,.$\smallskip

We choose the following norms on $E_1$
, $E'_1$:
\begin{align*}
\left\Vert h\right\Vert _{\mathcal{N}_{1}}: & =\left\Vert h\right\Vert _{\delta}+\Lambda_{1}^{-\frac{\vartheta}{\alpha}\left(\sigma-\delta\right)}\left\Vert h\right\Vert _{\sigma}\\
\left\Vert k\right\Vert'_{\mathcal{N}_{1}} & :=\left\Vert k\right\Vert _{\delta}'+\Lambda_{1}^{-\frac{\vartheta}{\alpha}\left(\sigma-\delta\right)}\left\Vert k\right\Vert _{\sigma}'
\end{align*}

Endowed with these norms, $E_1$ and $E'_1$ are Banach spaces.
We shall use the notation $||| \, L\, |||_{_{\mathcal{N}_{1}}}$ for the operator norm of any linear continuous map $L$ from the Banach space $E'_1$ to a Banach space that can be either $E_1$ or $E'_1$.\medskip

The map $F_{\varepsilon}$ induces a map $f_{1}:\,B_{s_0+m}\cap E_1\rightarrow E'_1$
defined by 
\[
f_{1}\left(u\right):=\Pi'_1 F_{\varepsilon}\left(u\right)
\]
for $u\in B_{s_0+m}\cap E_1$. Note that $f_{1}\left(0\right)=0$. We will use the local surjection
theorem to show that the range of $f_{1}$ covers a neighbourhood
of $0$ in $E'_1$. We begin by showing that $Df_{1}$ has a
right inverse.

{\it Note that, if we assume that $DF$ is tame Galerkin right-invertible, we can take $M_1=\Lambda_1\geq\underline{\Lambda}$, and $Df_1$
is automatically right-invertible, with the tame estimate (\ref{galtamegalinverseepsilon}). So the next subsection is only necessary if we assume that $DF$ is tame right-invertible.}

\subsubsection{\label{subsec:-i1}$Df_{1}(u)$ has a right inverse for $\Vert u\Vert_{\mathcal{N}_1}\leq 1\,$.}

This subsection uses condition (\ref{eq:i2}). We recall it here for
the reader's convenience:
\[
\left(1-\vartheta\right)\left(\sigma-\delta\right)>\vartheta m+\max\left\{ \ell,\vartheta {\ell}'\right\} +\frac{g}{\eta}
\]
\begin{lem}
\label{lem:i1} For $K$ large enough and for all $u\in E_1$ with $\left\Vert u\right\Vert _{\mathcal{N}_{1}}\leq 1$:
\[
|||\, \Pi'_1DF_{\varepsilon}\left(u\right)\left(1-\Pi_1\right)L_{\varepsilon}\left(u\right)|||_{_{\mathcal{N}_{1}}}\leq\frac{1}{2}
\]
\end{lem}

\begin{proof}
From $\left\Vert u\right\Vert _{\mathcal{N}_{1}}\leq1$, it follows
that $\left\Vert u\right\Vert _{\delta}\leq1$, and since $\delta>s_{0}+\max\left\{ \ell,m\right\} +\ell'$,
the tame estimates hold at $u$. 

Take any $k\in E_1'$ and set $h=\left(1-\Pi_1\right)L_{\varepsilon}\left(u\right)k$.

We have $\left\Vert h\right\Vert _{\delta}\lesssim \Lambda_{1}^{\delta-\sigma}\left\Vert L_{\varepsilon}\left(u\right)k\right\Vert _{\sigma}$,
and:
\begin{align*}
\left\Vert \Pi_1'DF_{\varepsilon}\left(u\right)h\right\Vert _{\delta-m}' & \lesssim\left\Vert h\right\Vert _{s_{0}+m}\left\Vert u\right\Vert _{\delta}+\left\Vert h\right\Vert _{\delta}\lesssim\left\Vert h\right\Vert _{\delta}\\
\left\Vert \Pi_1'DF_{\varepsilon}\left(u\right)h\right\Vert _{\delta}' & \lesssim M_{1}^{m}\left\Vert \Pi_1'DF_{\varepsilon}\left(u\right)h\right\Vert _{\delta-m}\lesssim M_{1}^{m}\left\Vert h\right\Vert _{\delta}\,.
\end{align*}

Hence:
\[
\left\Vert \Pi_1'DF_{\varepsilon}\left(u\right)h\right\Vert _{\delta}'\lesssim M_{1}^{m}\Lambda_{1}^{\delta-\sigma}\left\Vert L_{\varepsilon}\left(u\right)k\right\Vert _{\sigma}\,.
\]

Writing $\left\Vert \Pi_1'DF_{\varepsilon}\left(u\right)h\right\Vert _{\sigma}'\lesssim M_{1}^{\sigma-\delta}\left\Vert \Pi_1'DF_{\varepsilon}\left(u\right)h\right\Vert _{\delta}$ we
finally get:
\begin{equation}
\left\Vert \Pi_1'DF_{\varepsilon}\left(u\right)h\right\Vert _{\mathcal{N}_{1}}'\lesssim M_{1}^{m}\Lambda_{1}^{\delta-\sigma}\left(1+\Lambda_{1}^{-\frac{\vartheta}{\alpha}\left(\sigma-\delta\right)}M_{1}^{\sigma-\delta}\right)\left\Vert L_{\varepsilon}\left(u\right)k\right\Vert _{\sigma}\,.
\label{eq:i20}
\end{equation}

We now have to estimate $\left\Vert L_{\varepsilon}\left(u\right)k\right\Vert _{\sigma}$.
By the tame estimates, we have: 
\begin{align*}
\left\Vert L_{\varepsilon}\left(u\right)k\right\Vert _{\sigma} & \lesssim\varepsilon^{-g}\left(\left\Vert k\right\Vert _{\sigma+\ell'}'+\left\Vert u\right\Vert _{\sigma+\ell}\left\Vert k\right\Vert _{s_{0}+\ell'}'\right)\\
 & \lesssim\varepsilon^{-g}\left(M_{1}^{\ell'}\left\Vert k\right\Vert _{\sigma}'+\Lambda_{1}^{\ell}\left\Vert u\right\Vert _{\sigma}\left\Vert k\right\Vert _{\delta}'\right)
\end{align*}

Since $\left\Vert u\right\Vert _{\mathcal{N}_{1}}\leq1$, we have
$\left\Vert u\right\Vert _{\sigma}\leq \Lambda_{1}^{\frac{\vartheta}{\alpha}\left(\sigma-\delta\right)}\,$.
Substituting, we get:
\begin{align}
\left\Vert L_{\varepsilon}\left(u\right)k\right\Vert _{\sigma} & \lesssim\varepsilon^{-g}\left(M_{1}^{\ell'}\left\Vert k\right\Vert _{\sigma}'+\Lambda_{1}^{\frac{\vartheta}{\alpha}\left(\sigma-\delta\right)+\ell}\left\Vert k\right\Vert _{\delta}'\right)\nonumber \\
 & \lesssim\varepsilon^{-g}\Lambda_{1}^{\frac{\vartheta}{\alpha}\left(\sigma-\delta\right)}\left ( M_{1}^{\ell'}+\Lambda_{1}^{\ell}\right) \left\Vert k\right\Vert _{\mathcal{N}_{1}}'\label{eq:i21}
\end{align}

Putting (\ref{eq:i20}) and (\ref{eq:i21}) together, we get:
\[
\text{\ensuremath{\left\Vert \Pi_1'DF_{\varepsilon}\left(u\right)h\right\Vert _{\mathcal{N}_{1}}'\lesssim\varepsilon^{-g}M_{1}^{m}\Lambda_{1}^{\delta-\sigma}\left(\Lambda_{1}^{\frac{\vartheta}{\alpha}\left(\sigma-\delta\right)}+M_{1}^{\sigma-\delta}\right)\left ( M_{1}^{\ell'}+\Lambda_{1}^{\ell}\right) \left\Vert k\right\Vert _{\mathcal{N}_{1}}'}}
\]

Since $\alpha>1$, we have $\Lambda_{1}^{\frac{\vartheta}{\alpha}\left(\sigma-\delta\right)}\leq \Lambda_{1}^{\vartheta\left(\sigma-\delta\right)}= M_{1}^{\sigma-\delta}$,
so that:

\begin{align*}
\left\Vert \Pi_1'DF_{\varepsilon}\left(u\right)h\right\Vert _{\mathcal{N}_{1}}' & \lesssim\varepsilon^{-g}M_{1}^{m}\Lambda_{1}^{\delta-\sigma}M_{1}^{\sigma-\delta}\left ( M_{1}^{\ell'}+\Lambda_{1}^{\ell}\right) \left\Vert k\right\Vert _{\mathcal{N}_{1}}'\\
 & \lesssim\varepsilon^{-g}\Lambda_{1}^{\vartheta m-\left(1-\vartheta\right)\left(\sigma-\delta\right)+\max\{ \ell,\vartheta \ell'\} }\left\Vert k\right\Vert _{\mathcal{N}_{1}}'
\end{align*}

Since $\Lambda_{1}=K \varepsilon^{-\eta}$,
the inequality becomes:
\[
\left\Vert \Pi_1'DF_{\varepsilon}\left(u\right)h\right\Vert _{\mathcal{N}_{1}}'\lesssim K^{-C_0}\varepsilon^{-g+\eta C_0 }\left\Vert k\right\Vert _{\mathcal{N}_{1}}'
\]
with $C_0:=\left(1-\vartheta\right)\left(\sigma-\delta\right)-\vartheta m-\max\{\ell,\vartheta \ell'\}$.

By condition (\ref{eq:i2}), the exponent $C_0$ is larger than $g/\eta$, and the proof follows by choosing $K$ large enough independently of $0<\varepsilon\leq 1$.
\end{proof}
Introduce the map $\mathcal{L}_{1}\left(u\right)=\Pi_1 L_{\varepsilon}\left(u\right)_{\vert_{E'_1}}$.
Since $DF_{\varepsilon}\left(u\right)L_{\varepsilon}\left(u\right)=1\text{, }$it
follows from Lemma \ref{lem:i1} that, for $k\in E_1'$, $u\in E_1$ and $\left\Vert u\right\Vert _{\mathcal{N}_{1}}\leq 1\,$, we have:
\[
\left\Vert k-Df_{1}\left(u\right)\mathcal{L}_{1}\left(u\right)k\right\Vert'_{\mathcal{N}_{1}}\leq\frac{1}{2}\left\Vert k\right\Vert'_{\mathcal{N}_{1}}
\]
 This implies that the Neumann series $\sum_{i\geq 0} \left(I_{E'_1}-Df_{1}\left(u\right)\mathcal{L}_{1}\left(u\right)\right)^i$ converges in operator norm. Its sum is $S_1(u)=\left(Df_{1}(u)\mathcal{L}_{1}(u)\right)^{-1}$ and it has operator norm at most $2$.

Then $T_{1}\left(u\right):=\mathcal{L}_1(u)S_1(u)$ is a right
inverse of $Df_{1}\left(u\right)$ and $|||\, T_{1}\left(u\right)|||_{_{\mathcal{N}_{1}}}\leq 2\,|||\, \mathcal{L}_{1}\left(u\right)|||_{_{\mathcal{N}_{1}}}$.
By the tame estimates, if $u\in E_1\,,$ $\left\Vert u\right\Vert _{\mathcal{N}_{1}}\leq 1$ and $k\in E_1'$, we have:

\begin{align*}
\left\Vert \mathcal{L}_{1}\left(u\right)k\right\Vert _{\delta}\lesssim\left\Vert L_{\varepsilon}\left(u\right)k\right\Vert _{\delta} & \lesssim\varepsilon^{-g}\left(\left\Vert k\right\Vert _{\delta+\ell'}'+\left\Vert u\right\Vert _{\delta+\ell}\left\Vert k\right\Vert _{s_{0}+\ell'}'\right)\\
 & \lesssim\varepsilon^{-g}\left(M_{1}^{\ell'}+\Lambda_{1}^{\ell}\right)\left\Vert k\right\Vert _{\delta}'
\end{align*}

Combining with (\ref{eq:i21}), we find:
\[
\sup_{\Vert u\Vert_{_{\mathcal{N}_1}\leq 1}} |||\,  T_{1}\left(u\right)|||_{_{\mathcal{N}_{1}}}\lesssim\varepsilon^{-g}\left(M_{1}^{\ell'}+\Lambda_{1}^{\ell}\right) = m_1
\]

\subsubsection{Local inversion of $f_{1}$.}
\leavevmode\par

\noindent
Applying Theorem 5, we find that if $\left\Vert \Pi'_0 v\right\Vert'_{\mathcal{N}_{1}} < 1/m_1$,
then equation $f_{1}\left(u\right)=\Pi'_0 v$ has a solution $u_1 \in E_1$ with $\left\Vert u_1\right\Vert _{\mathcal{N}_{1}}\leq 1$ and $\left\Vert u_1\right\Vert _{\mathcal{N}_{1}}\leq m_1 \left\Vert \Pi'_0 v\right\Vert'_{\mathcal{N}_{1}}$. 

Note that $\left\Vert \Pi'_0 v\right\Vert _{\sigma}'\lesssim M_{0}^{\sigma-\delta}\left\Vert \Pi'_0 v\right\Vert _{\delta}'\lesssim \Lambda_{1}^{\frac{\vartheta}{\alpha}\left(\sigma-\delta\right)}\left\Vert \Pi'_0 v\right\Vert _{\delta}'$.
It follows that 
\[
\left\Vert \Pi'_0 v\right\Vert' _{\mathcal{N}_{1}}=\left\Vert \Pi'_0 v\right\Vert _{\delta}'+\Lambda_{1}^{-\frac{\vartheta}{\alpha}\left(\sigma-\delta\right)}\left\Vert \Pi'_0 v\right\Vert _{\sigma}'\lesssim\left\Vert \Pi'_0 v\right\Vert _{\delta}
\]
\medskip

Assume from now on: 
\begin{equation}
\left\Vert v\right\Vert _{\delta}'\lesssim\varepsilon^{g}\left(M_{1}^{\ell'}+\Lambda_{1}^{\ell}\right)^{-1}\label{v_petit}
\end{equation}

Then $\left\Vert \Pi'_0 v\right\Vert'_{\mathcal{N}_{1}}\lesssim m_1^{-1}$,
and Theorem 5 applies. The estimate on $u_1$ implies:
\begin{equation}
\left\Vert u_1\right\Vert _{\delta} \lesssim\text{\ensuremath{\varepsilon^{-g}\left(M_{1}^{\ell'}+\Lambda_{1}^{\ell}\right)\left\Vert v\right\Vert _{\delta}'}}\leq 1\label{u_1_delta}
\end{equation}

It also implies an estimate in higher norm:
\begin{equation}
\left\Vert u_1\right\Vert _{\sigma} \lesssim\varepsilon^{-g}\Lambda_{1}^{\frac{\vartheta}{\alpha}\left(\sigma-\delta\right)}\left(M_{1}^{\ell'}+\Lambda_{1}^{\ell}\right)\left\Vert v\right\Vert'_{\delta} \lesssim\, \Lambda_{1}^{\frac{\vartheta}{\alpha}\left(\sigma-\delta\right)}
\;.\label{u_1_sigma}
\end{equation}

\subsection{Induction.}

\subsubsection{Finding uniform bounds.}

In addition to $\left(\alpha,\vartheta,\delta,\varepsilon,\eta\right)$
we are given $\beta$ satisfying relations (\ref{eq:i4}) and (\ref{eq:i3}) .
We recall them here for the reader's convenience. With $s_{1}\geq s_{0}+\max\left\{ m,\ell\right\} $
and $\delta>s_{1}+\ell'$ ,
\begin{align*}
\sigma & >\alpha\beta+s_{1}\\
\beta & >\frac{1}{\alpha}\left(\sigma-\delta\right)
\end{align*}

We also inherit $\Lambda_{1}=K\varepsilon^{-\eta}$ and $u_1$ from the preceding section. Combining (\ref{eq:i3}) and (\ref{u_1_sigma}), we immediately obtain the estimate
\begin{equation}
\left\Vert u_1\right\Vert _{\sigma} \lesssim\varepsilon^{-g}\Lambda_{1}^{\beta}\left(M_{1}^{\ell'}+\Lambda_{1}^{\ell}\right)\left\Vert v\right\Vert'_{\delta} \,\lesssim \,\Lambda_{1}^{\beta}
\;.\label{u_1_sigma_beta}
\end{equation}

Consider
the sequences of integers $M_{n}$ and $\Lambda_{n}$, $n\geq1\text{, }$defined
by $\Lambda_{n}:=\Lambda_{1}^{\text{\ensuremath{\alpha^{n-1}}}}$
and $M_{n}:=\Lambda_{n}^{\vartheta}$. \medskip
Let $\Pi_n:=\Pi(\Lambda_n)\,,\;\Pi'_n:=\Pi'(M_n)\,,\;E_n:=E(\Lambda_n)\,,\;E'_n:=E'(M_n)\,.$\medskip

We will
construct a sequence $u_{n}\in E_{n},\,n\geq1,$ starting from the
initial point $u_{1}$ we found in the preceding section. 
For all $n\geq2$ the remaining points should satisfy the following
conditions: 
\begin{align}
\Pi_n'F_{\varepsilon}\left(u_{n}\right) & =\Pi_{n-1}'v\label{eq:i17}\\
\left\Vert u_{n}-u_{n-1}\right\Vert _{s_{0}} & \leq\varepsilon^{-g}\Lambda_{n-1}^{\alpha\beta-\sigma+s_{0}}\left(M_{1}^{\ell'}+\Lambda_{1}^{\ell}\right)\left\Vert v\right\Vert _{\delta}'\label{eq:i18}\\
\left\Vert u_{n}-u_{n-1}\right\Vert _{\sigma} & \leq\varepsilon^{-g}\Lambda_{n-1}^{\alpha\beta}\left(M_{1}^{\ell'}+\Lambda_{1}^{\ell}\right)\left\Vert v\right\Vert _{\delta}'\label{eq:i19}
\end{align}

We proceed by induction. Suppose we have found $u_{2},...,u_{n-1}$
satisfying these conditions. We want to construct $u_{n}.$ 
\begin{lem}
\label{lem:i5}Let us impose $K\geq 2$. For all $t$ with $s_{0}\leq t <\sigma-\alpha\beta$,
and all $i$ with $2\leq i\leq n-1$, we have:
\[
\sum_{i=2}^{n-1}\left\Vert u_{i}-u_{i-1}\right\Vert _{t}\leq \varepsilon^{-g}\left(M_{1}^{\ell'}+\Lambda_{1}^{\ell}\right)\Sigma\left(t\right)\left\Vert v\right\Vert _{\delta}'
\]
where $\Sigma\left(t\right)$ is finite and independent of $n\,,\,\varepsilon$.
\end{lem}

\begin{proof}
By the interpolation formula,
$$\left\Vert u_{i}-u_{i-1}\right\Vert _{t}\leq\varepsilon^{-g}\,\Lambda_{i-1}^{\alpha\beta-\sigma+t}\left(M_{1}^{\ell'}+\Lambda_{1}^{\ell}\right)\left\Vert v\right\Vert _{\delta}'$$
for all $2\leq i\leq n\,$. Since $\Lambda_{1}=K\varepsilon^{-\eta}\geq 2$, we have:
\begin{align*}
\sum_{i=2}^{n-1}\left\Vert u_{i}-u_{i-1}\right\Vert _{t} & \leq\varepsilon^{-g}\left(M_{1}^{\ell'}+\Lambda_{1}^{\ell}\right)\sum_{i=2}^{\infty}\Lambda_{i-1}^{\alpha\beta-\sigma+t}\left\Vert v\right\Vert _{\delta}'\\
 & \leq\varepsilon^{-g} \left(M_{1}^{\ell'}+\Lambda_{1}^{\ell}\right)\sum_{j=0}^{\infty} 2^{\alpha^{j(\alpha\beta-\sigma+t)}}\left\Vert v\right\Vert _{\delta}'
\end{align*}
 
\end{proof}
By (\ref{eq:i4}) we can take $t=s_{1}$, and we find a uniform bound
for $u_{n-1}$ in the $s_{1}$-norm,
namely:
\begin{align*}
\left\Vert u_{n-1}\right\Vert _{s_{1} } & \leq\left\Vert u_{1}\right\Vert _{\delta}+\sum_{i=2}^{n-1}\left\Vert u_{i}-u_{i-1}\right\Vert _{s_{1} }\\
 & \lesssim\varepsilon^{-g}\left(M_{1}^{\ell'}+\Lambda_{1}^{\ell}\right)\left(1+\Sigma(s_{1})\right)\left\Vert v\right\Vert _{\delta}'\\
 & \lesssim\varepsilon^{-g}\left(M_{1}^{\ell'}+\Lambda_{1}^{\ell}\right)\left\Vert v\right\Vert _{\delta}'
\end{align*}

In particular, we will have $\left\Vert u_{n-1}\right\Vert _{s_{1} }\leq 1$
if $\left\Vert v\right\Vert _{\delta}'\lesssim\varepsilon^{g}$$\left(M_{1}^{\ell'}+\Lambda_{1}^{\ell}\right)^{-1}$,
so the tame estimates hold at $u_{n-1}.$ 

Similarly, if $\left\Vert v\right\Vert _{\delta}'\lesssim\varepsilon^{g}$$\left(M_{1}^{\ell'}+\Lambda_{1}^{\ell}\right)^{-1}$
we find a uniform bound in the $\sigma$-norm. We have:

$$\left\Vert u_{n-1}\right\Vert _{\sigma} \leq\left\Vert u_{1}\right\Vert _{\sigma}+\sum_{i=2}^{n-1}\left\Vert u_{i}-u_{i-1}\right\Vert _{\sigma}\nonumber$$
and
$$\sum_{i=2}^{n-1}\left\Vert u_{i}-u_{i-1}\right\Vert _{\sigma}\lesssim\text{\ensuremath{\varepsilon^{-g}\left(M_{1}^{\ell'}+\Lambda_{1}^{\ell}\right)\sum_{i=1}^{n-1}\Lambda_{i}^{\beta}\left\Vert v\right\Vert _{\delta}'}}\lesssim\varepsilon^{-g}\left(M_{1}^{\ell'}+\Lambda_{1}^{\ell}\right)\Lambda_{n-1}^{\beta}\left\Vert v\right\Vert _{\delta}'$$
so, combining this with (\ref{u_1_sigma_beta}), we get:
\begin{equation}
\left\Vert u_{n-1}\right\Vert _{\sigma}\lesssim\varepsilon^{-g}\Lambda_{n-1}^{\beta}\left(M_{1}^{\ell'}+\Lambda_{1}^{\ell}\right)\left\Vert v\right\Vert _{\delta}'\,\lesssim\, \Lambda_{n-1}^{\beta}\,.\label{eq:sigma-estim}
\end{equation}

\subsubsection{Setting up the induction step.}
\leavevmode\par

\noindent
Suppose, as above, that $\left\Vert v\right\Vert _{\delta}'\lesssim\varepsilon^{g}\left(M_{1}^{\ell'}+\Lambda_{1}^{\ell}\right)^{-1}$and
that $u_{2},...,u_{n-1}$ have been found. We have seen that $\left\Vert u_{n-1}\right\Vert_{s_1} \leq 1$,
so that the tame estimates hold at $u_{n-1}$, and we also have $\left\Vert u_{n-1}\right\Vert _{\sigma}\lesssim \Lambda_{n-1}^{\beta}$
. We want to find $u_{n}$ satisfying (\ref{eq:i17}), (\ref{eq:i18})
and (\ref{eq:i19}). Since $\Pi'_{n-1} F_{\varepsilon}\left(u_{n}\right)=\Pi'_{n-2} v$
, we rewrite the latter equation as follows:
\begin{equation}
\Pi'_{n} \left(F_{\varepsilon}\left(u_{n}\right)-F_{\varepsilon}\left(u_{n-1}\right)\right)+\left(\Pi'_{n}-\Pi'_{n-1}\right) F_{\varepsilon}\left(u_{n-1}\right)=\left(\Pi'_{n-1}-\Pi'_{n-2}\right)v\label{eq:i22}
\end{equation}

Define a map $f_{n}:\,E_{n}\rightarrow E'_{n}$ with $f_n\left(0\right)=0$
by:
\[
f_{n}\left(z\right)=\Pi'_{n}\left(F_{\varepsilon}\left(u_{n-1}+z\right)-F_{\varepsilon}\left(u_{n-1}\right)\right)
\]

Equation (\ref{eq:i22}) can be rewritten as follows:
\begin{align}
f_{n}\left(z\right) & =\Delta_{n}v+e_{n}\label{eq:i23}\\
\Delta_{n}v & =\Pi'_{n-1}\left(1-\Pi'_{n-2}\right)v\label{eq:i24}\\
e_{n} & =-\Pi'_{n}\left(1-\Pi'_{n-1}\right)F_{\varepsilon}\left(u_{n-1}\right)\label{eq:i25}
\end{align}

We choose the following norms on $E_{n}$ and $E'_{n}$:
\begin{align*}
\left\Vert x\right\Vert _{\mathcal{N}_{n}} & =\left\Vert x\right\Vert _{s_{0}}+\Lambda_{n-1}^{-\sigma+s_0}\left\Vert x\right\Vert _{\sigma}\\
\left\Vert y\right\Vert _{\mathcal{N}_{n}}' & =\left\Vert y\right\Vert _{s_{0}}'+\Lambda_{n-1}^{-\sigma+s_{0}}\left\Vert y\right\Vert _{\sigma}'
\end{align*}

Endowed with these norms, $E_n$ and $E'_n$ are Banach spaces.
We shall use the notation $||| \, L\, |||_{_{\mathcal{N}_{n}}}$ for the operator norm of any linear continuous map $L$ from the Banach space $E'_n$ to a Banach space that can be either $E_n$ or $E'_n$.\medskip

\begin{lem}
\label{lem:i2}If $0\leq t\leq\sigma-s_{0}$, then:
\begin{align*}
\left\Vert x\right\Vert _{s_{0}+t} & \lesssim \Lambda_{n-1}^{t}\left\Vert x\right\Vert _{\mathcal{N}_{n}}\\
\left\Vert y\right\Vert'_{s_{0}+t} & \lesssim \Lambda_{n-1}^{t}\left\Vert y\right\Vert _{\mathcal{N}_{n}}'
\end{align*}
\end{lem}

\begin{proof}
Use the interpolation inequality.
\end{proof}

We will solve the system (\ref{eq:i23}), (\ref{eq:i24}), (\ref{eq:i25})
by applying the local surjection theorem to $f_{n}$ on the ball $B_{\mathcal{N}_{n}}\left(0,r_{n}\right)\subset E_{n}$ where:
\begin{equation}
r_{n}=\varepsilon^{-g}\Lambda_{n-1}^{\alpha\beta-\sigma+s_{0}}\left(M_{1}^{\ell'}+\Lambda_{1}^{\ell}\right)\left\Vert v\right\Vert _{\delta}'\label{eq:i29}
\end{equation}

Note that if the solution $z$ belongs to $B_{\mathcal{N}_{n}}\left(0,r_{n}\right)$, then
$$\left\Vert z\right\Vert _{s_{0}}\leq\varepsilon^{-g}\Lambda_{n-1}^{\alpha\beta-\sigma+s_{0}}\left(M_{1}^{\ell'}+\Lambda_{1}^{\ell}\right)\left\Vert v\right\Vert _{\delta}'\ 
\hbox{ and }\ \left\Vert z\right\Vert _{\sigma}\leq\varepsilon^{-g}\Lambda_{n-1}^{\alpha\beta}\left(M_{1}^{\ell'}+\Lambda_{1}^{\ell}\right)\left\Vert v\right\Vert _{\delta}'\,.$$
In other words, $u_{n}=u_{n-1}+z$ satisfies (\ref{eq:i18}) and
(\ref{eq:i19}), so that the induction step is proved.

We begin by showing that $Df_{n}\left(z\right)$ has a right inverse.

{\it Note that, if we assume that $DF_\varepsilon$ is tame Galerkin right-invertible, we can take $M_n=\Lambda_n$, and the result of the next subsection is obvious. This subsection is only useful if we assume that $DF$ is tame right-invertible but not tame Galerkin right-invertible.}

\subsubsection{$Df_{n}(z)$ has a right inverse for $\Vert z\Vert_{\mathcal{N}_n}\leq r_n\,$.\label{subsec:i2}}

In this subsection, we use conditions (\ref{eq:i5}) and (\ref{eq:i6}).
We recall them for the reader's convenience:

\begin{align*}
\left(1+\alpha-\vartheta\alpha\right)\left(\sigma-s_{0}\right) & >\alpha\beta+\alpha\left(m+\ell\right)+\ell'+\frac{g}{\eta}\\
\left(1-\vartheta\right)\left(\sigma-s_{0}\right) & >m+\vartheta \ell'+\frac{g}{\alpha\eta}
\end{align*}

Take now any $z\in B_{\mathcal{N}_{n}}\left(0,r_{n}\right)$. Arguing as above, we find that if , then:
\begin{align}
\label{bc}\left\Vert u_{n-1}+z\right\Vert _{s_{1} } & \leq 1    \\
\label{bd}\left\Vert u_{n-1}+z\right\Vert _{\sigma} & \lesssim \Lambda_{n}^{\beta}
\end{align}

By (\ref{bc}) the tame estimates hold on $z\in B_{\mathcal{N}_{n}}\left(0,r_{n}\right)$.
\begin{lem}
Take $\Lambda_1= K \varepsilon^{-\eta}$ with $K>1$ chosen large enough, independently of $n$ and  $\varepsilon\in (0,1]$. Then, for all $z\in B_{\mathcal{N}_{n}}\left(0,r_{n}\right)$:
\[
|||\, \Pi'_n DF_{\varepsilon}\left(u_{n-1}+z\right)\left(1-\Pi_{n}\right)L_{\varepsilon}\left(u_{n-1}+z\right)|||_{ _{\mathcal{N}_{n}}}\leq\frac{1}{2}
\]
\end{lem}

\begin{proof}
We proceed as in the proof of Lemma \ref{lem:i1}. For $k\in E'_n$,
we set 
\[h=\left(1-\Pi_{n}\right)L_{\varepsilon}\left(u_{n-1}+z\right)k\;.\]
We have:
\[
\left\Vert h\right\Vert _{s_{0}+m}\lesssim \Lambda_{n}^{-\sigma+s_{0}+m}\left\Vert L_{\varepsilon}\left(u_{n-1}+z\right)k\right\Vert _{\sigma}
\]

By (\ref{bd})  and the tame estimates for $L_\varepsilon$, we get:
\begin{align}
\left\Vert L_{\varepsilon}\left(u_{n-1}+z\right)k\right\Vert _{\sigma} & \lesssim\varepsilon^{-g}\left(\left\Vert u_{n-1}+z\right\Vert _{\sigma+\ell}\left\Vert k\right\Vert _{s_{0}+\ell'}'+\left\Vert k\right\Vert _{\sigma+\ell'}'\right)\nonumber \\
 & \lesssim\varepsilon^{-g}\left(\Lambda_{n}^{\beta+\ell}\Lambda_{n-1}^{\ell'}+M_{n}^{l'}\Lambda_{n-1}^{\sigma-s_{0}}\right)\left\Vert k\right\Vert _{\mathcal{N}_{n}}'\label{eq:i27}
\end{align}
where we have used Lemma \ref{lem:i2}. Substituting in the preceding
formula, we get:
\[
\left\Vert h\right\Vert _{s_{0}+m}\lesssim\varepsilon^{-g}\left(\Lambda_{n}^{\beta+\ell-\sigma+s_{0}+m}\Lambda_{n-1}^{l'}+M_{n}^{\ell'}\Lambda_{n-1}^{-\left(\alpha-1\right)\left(\sigma-s_{0}\right)+\alpha m}\right)\left\Vert k\right\Vert _{\mathcal{N}_{n}}'
\]

By the tame estimate (\ref{tamedirectepsilon}), we have:
\[
\left\Vert \Pi'_n DF_{\varepsilon}\left(u_{n-1}+z\right)h\right\Vert _{s_{0}}'\lesssim\left\Vert h\right\Vert _{s_{0}+m}
\]

From this it follows that:
\[
\left\Vert \Pi'_n DF_{\varepsilon}\left(u_{n-1}+z\right)h\right\Vert _{\sigma}'\lesssim M_{n}^{\sigma-s_{0}}\left\Vert h\right\Vert _{s_{0}+m}
\]

Hence:
\[
\left\Vert \Pi'_n DF_{\varepsilon}\left(u_{n-1}+z\right)h\right\Vert _{\mathcal{N}_{n}}'\lesssim\left(1+\Lambda_{n-1}^{-\sigma+s_{0}}M_{n}^{\sigma-s_{0}}\right)\left\Vert h\right\Vert _{s_{0}+m}
\]

We have $\Lambda_{n-1}^{-\sigma+s_{0}}M_{n}^{\sigma-s_{0}}\lesssim \Lambda_{n-1}^{\left(\alpha\vartheta-1\right)\left(\sigma-s_{0}\right)}$.
Since $\alpha\vartheta>1$, the dominant term in the parenthesis is
the second one, and:
\begin{align*}
&\left\Vert \Pi'_n DF_{\varepsilon}\left(u_{n-1}+z\right)h\right\Vert _{\mathcal{N}_{n}} \\
&\ \ \ \ \ \ \ \ \ \ \ \ \ \ \ \ \lesssim \Lambda_{n-1}^{-\sigma+s_{0}}M_{n}^{\sigma-s_{0}}\left\Vert h\right\Vert _{s_{0}+m}\\
 &\ \ \ \ \ \ \ \ \ \ \ \ \ \ \ \ \lesssim\varepsilon^{-g}M_{n}^{\sigma-s_{0}}\left(\Lambda_{n-1}^{\alpha\left(\beta+\ell-\sigma+s_{0}+m\right)+\ell'-\sigma+s_{0}}+M_{n}^{\ell'}\Lambda_{n-1}^{-\alpha\left(\sigma-s_{0}\right)+\alpha m}\right)\left\Vert k\right\Vert _{\mathcal{N}_{n}}'
\end{align*}

From (\ref{eq:i5}) and (\ref{eq:i6}), it follows that the right-hand side is a decreasing function of $n$. To check that
it is less than $1/2$ for all $n\geq2$, it is enough to check it
for $n=2.$ Since $\Lambda_{1}=K\varepsilon^{-\eta},$
substituting in the right-hand side, we get:

$$\left\Vert \Pi'_n DF_{\varepsilon}\left(u_{n-1}+z\right)h\right\Vert _{\mathcal{N}_{n}}  \lesssim \left(K^{-\min\left\{ C_{1},C_{2}\right\}}\right)^{\alpha^{n-2}}\left(\varepsilon^{\min\left\{ C_{1},C_{2}\right\}-\alpha^{2-n}g/\eta }\right)^{\eta\alpha^{n-2}}\left\Vert k\right\Vert _{\mathcal{N}_{n}}'$$
with
\begin{align*}
C_{1} & =-\alpha(\beta+\ell+m)-\ell'+\left(1+\alpha-\alpha\vartheta\right)\left(\sigma-s_{0}\right)\\
C_{2} & =\alpha\left((1-\vartheta)(\sigma-s_{0})-\vartheta \ell'-m\right)
\end{align*}
 By (\ref{eq:i5}) and (\ref{eq:i6}), both exponents $C_{1}\text{ and }\ensuremath{C_{2}}$ are larger than $g/\eta$.
As a consequence, $\left\Vert \Pi'_n DF_{\varepsilon}\left(u_{n-1}+z\right)h\right\Vert _{\mathcal{N}_{n}}\leq \frac{1}{2}\left\Vert k\right\Vert _{\mathcal{N}_{n}}'$
for $K$ chosen large enough, independently of $n$ and $0<\varepsilon\leq 1$. 
\end{proof}
Define $\mathcal{L}_{n}\left(z\right)=\Pi_n L_{\varepsilon}\left(u_{n-1}+z\right)_{\vert_{E'_n}}\,.$
Arguing as in subsection \ref{subsec:-i1}, we find that the Neumann series $\sum_{i\geq 0} \left(I_{E'_n}-Df_{n}\left(u\right)\mathcal{L}_{n}\left(u\right)\right)^i$ converges in operator norm. Its sum is $S_n(u)=\left(Df_{n}(u)\mathcal{L}_{n}(u)\right)^{-1}$ and it has operator norm at most $2$. 
Then $T_{n}\left(u\right):=\mathcal{L}_n(u)S_n(u)$ is a right
inverse of $Df_{n}\left(u\right)\,,$ with the estimate $|||\, T_{n}\left(u\right)|||_{_{\mathcal{N}_{n}}}\leq 2\,|||\, \mathcal{L}_{n}\left(u\right)|||_{_{\mathcal{N}_{n}}}$.

We have already derived estimate (\ref{eq:i27}) which immediately implies:
\[
\left\Vert \mathcal{L}_{n}\left(z\right)k\right\Vert _{\sigma}\lesssim\varepsilon^{-g}\Lambda_{n-1}^{\sigma-s_{0}}\left(\Lambda_{n}^{\beta+\ell}\Lambda_{n-1}^{-\sigma+s_{0}+\ell'}+M_{n}^{\ell'}\right)\left\Vert k\right\Vert _{\mathcal{N}_{n}}'
\]

From the tame estimates and Lemma \ref{lem:i2}, we also have:
\[
\left\Vert \mathcal{L}_{n}\left(z\right)k\right\Vert _{s_{0}}\lesssim\varepsilon^{-g}\left\Vert k\right\Vert _{s_{0}+\ell'}'\lesssim\varepsilon^{-g}\Lambda_{n-1}^{\ell'}\left\Vert k\right\Vert _{\mathcal{N}_{n}}'
\]

Since $\alpha\vartheta>1$, we have $\Lambda_{n-1}^{\ell'}\lesssim M_{n}^{\ell'}$. So the two preceding estimates can be combined, and 
we get the final estimate for the right inverse in operator norm:
\begin{equation}
|||\, T_{n}\left(z\right)|||_{_{\mathcal{N}_{n}}}\lesssim\varepsilon^{-g}\left(\Lambda_{n}^{\beta+\ell}\Lambda_{n-1}^{-\sigma+s_0+\ell'}+M_{n}^{\ell'}\right)\label{eq:i28}
\end{equation}

\subsubsection{Finding $u_{n}$.}

In this subsection, we use relations (\ref{eq:i4}),(\ref{eq:i7}),
(\ref{eq:i8}) and (\ref{eq:i9}). We recall them for the reader's
convenience:
\begin{align*}
\sigma & >\alpha\beta+s_1\\
\delta & >s_{0}+\frac{\alpha}{\vartheta}\left(\sigma-s_{0}-\alpha\beta+\ell"\right)\\
\left(\alpha-1\right)\beta & >\left(1-\vartheta\right)\left(\sigma-s_{0}\right)+\vartheta m+\ell"+\frac{g}{\eta}\\
\ell" & =\max\left\{ \left(\alpha-1\right)\ell+\ell',\alpha\vartheta \ell'\right\} 
\end{align*}

Let us go back to (\ref{eq:i23}). By Theorem \ref{thm1}
to solve $\Pi'_n f_{n}\left(z\right)=\Delta_{n}v+e_{n}$ with
$z\in B_{\mathcal{N}_{n}}\left(0,r_{n}\right)$ it is enough that:
\begin{equation}
|||\, T_{n}\left(z\right)|||_{_{\mathcal{N}_{n}}}\left(\left\Vert \Delta_{n}v\right\Vert _{\mathcal{N}_{n}}+\left\Vert e_{n}\right\Vert _{\mathcal{N}_{n}}\right)\leq r_{n}\label{eq:i33}
\end{equation}

Here $r_{n}$ is given by (\ref{eq:i29}). We can estimate $|||\, T_{n}\left(z\right)|||_{_{\mathcal{N}_{n}}}$
using (\ref{eq:i28}). We need
to estimate $\left\Vert \Delta_{n}v\right\Vert _{\mathcal{N}_{n}}$ and
$\left\Vert e_{n}\right\Vert _{\mathcal{N}_{n}}$. 

From (\ref{eq:i24}) we have:
\begin{align*}
\left\Vert \Delta v\right\Vert _{s_{0}}' & \lesssim M_{n-2}^{s_{0}-\delta}\left\Vert v\right\Vert _{\delta}'\\
\left\Vert \Delta v\right\Vert _{\sigma}' & \lesssim M_{n-1}^{\sigma-\delta}\left\Vert v\right\Vert _{\delta}'\\
\left\Vert \Delta v\right\Vert _{\mathcal{N}_{n}}' & \lesssim\max\left\{ M_{n-2}^{s_{0}-\delta},\Lambda_{n-1}^{-\sigma+s_{0}}M_{n-1}^{\sigma-\delta}\right\} \left\Vert v\right\Vert _{\delta}'
\end{align*}

An easy calculation yields:
\[
\sigma - s_0 - \vartheta (\sigma - \delta) + \frac{\vartheta}{\alpha}(s_0 -\delta) = (1-\vartheta) (\sigma - \delta) + (1-\frac{\vartheta}{\alpha}) (\delta - s_0) \\
\]

Since $s_0<\delta<\sigma$ and $\vartheta<1<\alpha$, the two terms on the right-hand side are positive, so $\Lambda_{n-1}^{-\sigma+s_{0}}M_{n-1}^{\sigma-\delta}\lesssim M_{n-2}^{s_{0}-\delta}$.  It follows that:
\begin{equation}
\left\Vert \Delta v\right\Vert _{\mathcal{N}_{n}}'\lesssim M_{n-2}^{s_{0}-\delta}\left\Vert v\right\Vert _{\delta}'\label{eq:i30}
\end{equation}

From (\ref{eq:i25}), we derive: 
\[
\left\Vert e_{n}\right\Vert _{s_{0}}'\lesssim M_{n-1}^{-\sigma+m+s_{0}}\left\Vert e_{n}\right\Vert _{\sigma-m}'
\]

By Lemma \ref{boundF}, $\left\Vert F_\varepsilon\left(u_{n-1}\right)\right\Vert _{\sigma-m}\lesssim\left\Vert u_{n-1}\right\Vert _{\sigma}$
. So, remembering (\ref{eq:i25}) and (\ref{eq:sigma-estim}), we get:
\begin{align*}
\left\Vert e_{n}\right\Vert'_{s_{0}} & \lesssim M_{n-1}^{-\sigma+m+s_{0}}\left\Vert u_{n-1}\right\Vert _{\sigma}\\
 & \lesssim\varepsilon^{-g}M_{n-1}^{-\sigma+m+s_{0}}\Lambda_{n-1}^{\beta}\left(M_{1}^{\ell'}+\Lambda_{1}^{\ell}\right)\left\Vert v\right\Vert _{\delta}'
\end{align*}

Similarly, 
\begin{align*}
\left\Vert e_{n}\right\Vert'_\sigma & \lesssim\left\Vert u_{n-1}\right\Vert _{\sigma+m}\lesssim \Lambda_{n-1}^{m}\left\Vert u_{n-1}\right\Vert _{\sigma}\\
 & \lesssim\varepsilon^{-g}\Lambda_{n-1}^{\beta+m}\left(M_{1}^{\ell'}+\Lambda_{1}^{\ell}\right)\left\Vert v\right\Vert _{\delta}'
\end{align*}

Finally, since $M_{n-1}<\Lambda_{n-1}$ and $\sigma>m+s_{0}$ , we get:
\begin{equation}
\left\Vert e_{n}\right\Vert _{\mathcal{N}_{n}}'\lesssim\varepsilon^{-g}\Lambda_{n-1}^{\beta}M_{n-1}^{-\sigma+m+s_{0}}\left(M_{1}^{\ell'}+\Lambda_{1}^{\ell}\right)\left\Vert v\right\Vert _{\delta}'
\label{eq:i73}
\end{equation}

Substituting \eqref{eq:i28}, \eqref{eq:i29}, \eqref{eq:i30}, \eqref{eq:i73} in \eqref{eq:i33},  we get the following sufficient condition:
\begin{equation}
\left(\Lambda_{n}^{\beta+\ell}\Lambda_{n-1}^{-\sigma+s_{0}+\ell'}+M_{n}^{\ell'}\right)\left(M_{n-2}^{s_{0}-\delta}+\varepsilon^{-g}\Lambda_{n-1}^{\beta}M_{n-1}^{-\sigma+m+s_{0}}\right)\lesssim \Lambda_{n-1}^{\alpha\beta-\sigma+s_{0}}\label{eq:i31}
\end{equation}

We estimate both sides separately. Remembering that  $M_{n-i}=\left(\Lambda_{n-1}\right)^{\alpha^{1-i}\vartheta}$ and $\Lambda_{n-1}=\left(K\varepsilon^{-\eta}\right)^{\alpha^{n-2}}$, we find
\begin{align*}
\left(\Lambda_{n}^{\beta+\ell}\Lambda_{n-1}^{-\sigma+s_{0}+\ell'}+M_{n}^{\ell'}\right) \Big(M_{n-2}^{s_{0}-\delta} +\varepsilon^{-g} \Lambda_{n-1}^{\beta} & M_{n-1}^{-\sigma+m+s_{0}} \Big)\\
&\lesssim \Big(\varepsilon^{-\eta\alpha^{n-2}}\Big)^{\max\{C_3,C_4\}+\max\{C_5,C_6\}}
\end{align*}
and
$$\Lambda_{n-1}^{\alpha\beta-\sigma+s_{0}}\gtrsim \Big(\varepsilon^{-\eta\alpha^{n-2}}\Big)^{C_7}$$
with
\begin{align*}
C_3:=&\alpha(\beta+\ell)-\sigma+s_{0}+\ell'\\
C_4:=&\alpha\vartheta\ell'\\
C_5:=&\vartheta\alpha^{-1}(s_{0}-\delta)\\
C_6:=&g/\eta + \beta+ \vartheta(-\sigma+m+s_{0})\\
C_7:=&\alpha\beta-\sigma+s_{0}
\end{align*}

By (\ref{eq:i4}), we have $\sigma-\alpha\beta> s_1>s_{0}+\max\left\{ m,\ell\right\} $.
It follows that:
$$C_3 <\left(\alpha-1\right)\ell+\ell'\,.$$

So, defining $\ell"=\max\left\{ \left(\alpha-1\right)\ell+\ell',\alpha\vartheta \ell'\right\}$ as in (\ref{eq:i9}), we see that
$$\max\{C_3,C_4\}+\max\{C_5,C_6\}\leq \max\{\ell"+C_5,\ell"+C_6\} $$

So condition (\ref{eq:i31})
is implied by the inequalities $\ell"+C_5< C_7$ and $\ell"+C_6<C_7$, which are the same as conditions (\ref{eq:i7}) and (\ref{eq:i8}). So inequality \eqref{eq:i33} holds, and the induction holds by Theorem \ref{thm1}

\subsection{End of proof}

First of all, for the above construction to work, the only constraint on $S$ is $S>\sigma$, and Lemma \ref{compatibility} gives us the estimate $\sigma<{\zeta}g/\eta$. The constant $\eta$ is only constrained by condition \eqref{eq:i0}, and we can choose, for instance, $\eta=\frac{g'-g}{2\max\left\{ \vartheta \ell',\ell\right\}}$. So we only need a condition on $S$ of the form $S\geq S_0$ with $S_0=O(\frac{1}{g'-g})$ as $g'\to g\,,$ all the other parameters being fixed.\medskip

Let us now check that the estimate $\left\Vert v\right\Vert _{\delta}'\lesssim\varepsilon^{g'}$ is sufficient for the above construction. In \eqref{v_petit} we made the assumption $\left\Vert v\right\Vert _{\delta}'\lesssim\varepsilon^{g}\left(\Lambda_{1}^{\ell}+M_{1}^{\ell'}\right)^{-1}$ on $v\,,$
and we have $M_{1}\lesssim \varepsilon^{-\vartheta\eta}\,,\; \Lambda_{1}\lesssim\varepsilon^{-\eta}$, hence $\left(\Lambda_{1}^{\ell}+M_{1}^{\ell'}\right)\lesssim
\varepsilon^{-\eta\max\left\{ \vartheta \ell',\ell\right\} }\,.$ So the condition $\left\Vert v\right\Vert _{\delta}'\lesssim\varepsilon^{g+\eta\max\left\{ \vartheta \ell',\ell\right\} }$ guarantees the existence of the sequence $(u_n)$.
But (\ref{eq:i0}) may be rewritten in the form
$$g+\eta\max\left\{ \vartheta \ell',\ell\right\}<g'\,,$$
so the preceding condition is implied by the estimate $\left\Vert v\right\Vert _{\delta}'\lesssim\varepsilon^{g'}\,,$ which is thus sufficient,
as desired.\medskip

Now we can translate the symbol $\,\lesssim\,$ into more explicit estimates. Choosing $r>0$ small enough, our construction gives, for every $v\in W_{\delta}$ with $\left\Vert v\right\Vert _{\delta}'\leq r\,\varepsilon^{g'}$
a sequence $u_{n},\,n\geq 1$, such that $u_n\in E_n\,,$ $\left\Vert u_{n}\right\Vert _{s_{1}}\leq r^{-1}\varepsilon^{-g'}\Vert v\Vert _{\delta}'\leq 1\,,$
and
$$\Pi'_n F_\varepsilon\left(u_{n}\right)=\Pi'_{n-1}v\,.$$

It follows
from Lemma \ref{lem:i5} that for any $t<\sigma-\alpha\beta\,,$ $(u_{n})$ is a Cauchy sequence for
the $\Vert\cdot\Vert_t\,$. We recall
that, by condition (\ref{eq:i4}), $s_{1}<\sigma-\alpha\beta$. So we can choose $t_1\in (s_{1},\sigma-\alpha\beta)\,.$ Then $(u_n)$ converges to some $u_\varepsilon$ in $V_{t_1}$ with $\left\Vert u_\varepsilon\right\Vert _{s_1}\leq r^{-1}\varepsilon^{-g'}\Vert v\Vert _{\delta}'\leq 1\,.$\medskip

Since $t_1\geq s_{0}+m$, the map $F_\varepsilon$ is continuous
from the $t_1$-norm to the $\left(t_1-m\right)$-norm, so $F_\varepsilon\left(u_{n}\right)$
converges to $F_\varepsilon\left(u_\varepsilon\right)$ in $W_{t_1-m}\,.$ Then $F_\varepsilon\left(u_{n}\right)$ is a bounded sequence in $W_{t_1-m}$, and $t_1-m>s_0$. So, using 
the approximation estimate \eqref{gain}, we find that $\Vert(1-\Pi'_n)F_\varepsilon\left(u_{n}\right)\Vert_{s_0}\to 0\,,$ and finally $\Vert \Pi'_n F_\varepsilon\left(u_{n}\right)-F_\varepsilon(u_\varepsilon)\Vert_{s_0} \to 0$  as $n\to\infty\,.$\medskip

On the right-hand side, using \eqref{gain} again, we find that $\Pi'_{n-1} v$ converges
to $v$ in $W_{s_0}$, since $\delta > s_0\,$.\medskip

We conclude that $F_\varepsilon\left(u_\varepsilon\right)=v$, as desired, and this ends the proof of Theorem \ref{Thm8}.

\section{An application of the singular perturbation theorem}

\subsection{The result}

In this section, we consider a Cauchy problem for nonlinear Schr\"odinger
systems arising in nonlinear optics, a question recently studied by
M\'etivier-Rauch \cite{MR} and Texier-Zumbrun \cite{TZ}. M\'etivier-Rauch
proved the existence of local in time solutions, with an existence time $T$
converging to $0$ when the $H^{s}$ norm of the initial datum goes to infinity.
Texier-Zumbrun, thanks to their version of the Nash-Moser theorem adapted to
singular perturbation problems, were able to find a uniform lower bound on $T$
for certain highly concentrated initial data. The $H^{s}$ norm of these initial
data could go to infinity. By applying our "semiglobal" version of the
Nash-Moser theorem, we are able to extend Texier-Zumbrun's result to even
larger initial data. In the sequel we follow closely their
exposition, but some parameters are named differently to avoid confusions
with our other notations.\medskip

The problem takes the following form:

\begin{equation}
\label{Cauchy}
\left\{
    \begin{array}{ll}
    & \partial_{t}u+iA(\partial_{x})u=B(u,\partial_{x})u, \\
    & u(0,x)=\varepsilon^{\kappa} \left(a_\varepsilon(x), \bar{a_{\varepsilon}}(x)\right)
    \end{array}
\right.
\end{equation}
with $u(t,x)=(\psi(t,x),\bar{\psi}(t,x))\in\mathbf{C}^{2n}$, $(t,x)\in
\lbrack0,T]\times\mathbf{R}^{d}$,
\[
A(\partial_{x})=\mathrm{diag}(\lambda_{1},\cdots,\lambda_{n},-\lambda_{1},\cdots
,-\lambda_{n})\Delta_x%
\]
and
\[
B=%
\begin{pmatrix}
\mathcal{B} & \mathcal{C}\\
\bar{\mathcal{C}} & \bar{\mathcal{B}}%
\end{pmatrix}
\]
The coefficients $b_{jj^{\prime}},\ c_{jj^{\prime}}$ of the $n\times n$
matrices $\mathcal{B},\ \mathcal{C}$ are first-order operators with smooth
coefficients: $b_{jj^{\prime}}=\sum_{k=1}^{d}b_{kjj^{\prime}}(u)\partial
_{x_{k}}$, $c_{jj^{\prime}}=\sum_{k=1}^{d}c_{kjj^{\prime}}(u)\partial_{x_{k}}%
$, with $b_{kjj^{\prime}}$ and $c_{kjj^{\prime}}$ smooth complex-valued
functions of $u$ satisfying, for some integer $p\geq 2$, some $C>0$, all
$0\leq|\alpha|\leq p$ and all $u=(\psi,\bar{\psi})\in \mathbf{C}^{2n}$:
\[
|\partial^{\alpha}b_{kjj^{\prime}}(u)|+|\partial^{\alpha}c_{kjj^{\prime}}(u)|\leq
C|u|^{p-|\alpha|}\,.
\]
Moreover, we assume that the following ``transparency'' conditions hold: the
functions $b_{kjj}$ are real-valued, the coefficients $\lambda_{j}$ are real
and pairwise distinct, and for any $j,\,j^{\prime}$ such that $\lambda
_{j}+\lambda_{j^{\prime}}=0$, $c_{jj^{\prime}}=c_{j^{\prime}j}$. \medskip\ \ \ 

We consider initial data of the form $\varepsilon
^{\kappa}\left(a_{\varepsilon}(x),\bar{a_{\varepsilon}}(x)\right)$ with 
$a_{\varepsilon}(x)=a_{1}(x/\varepsilon)$ where $0<\varepsilon\leq 1$, $a_{1}\in H^{S}(\mathbf{R}^{d})$ for some $S$ large enough
and $\Vert a_1\Vert_{_{H^{S}}}$ small enough.
\medskip

Our
goal is to prove that the Cauchy problem has a solution on $[0,T]\times
\mathbf{R}^{d}$ for all $0<\varepsilon\leq 1\,$, with $T>0$ independent of
$\varepsilon$. Texier-Zumbrun obtain existence and uniqueness of the solution,
under some conditions on $\kappa$, which should be large enough. This
corresponds to a smallness condition on the initial datum when $\varepsilon$ approaches zero. Our local
surjection theorem only provides existence, but our condition on $\kappa$ is
less restrictive, so our initial datum is allowed to be larger. Note
that, once existence is proved, uniqueness is easily obtained for this Cauchy
problem, indeed local-in-time uniqueness implies global-in-time uniqueness. Our result is the following:

\begin{theorem}
\label{ThmTexier}

Under the above assumptions and notations, let us impose the additional condition
\begin{equation}
\label{kappa}
\kappa>\frac{d}{2(p-1)}\,.
\end{equation}

Let $s_1>\frac{d}{2}+4\,.$ If $0<\varepsilon\leq 1$, $a_1\in H^S(\mathbf{R}^d)$ for $S$ large enough, and $\Vert a_1\Vert_{H^{S}}$
is small enough, then the Cauchy problem (\ref{Cauchy}) has a unique solution in the functional space $C^1\left([0,T],\, H^{s_1-2}(\mathbf{R}^d)\right)\cap C^0\left([0,T],\, H^{s_1}(\mathbf{R}^d)\right)\,.$
\end{theorem}

Metivier-Rauch already
provide existence for a fixed positive $T$ when $\kappa\geq 1$ . So we obtain something new in comparison with them
when
$\frac{d}{2}\frac{1}{p-1}<1\,,$
that is, when
\[
p>1+\frac{d}{2}\;.
\]
Let us now compare our results with those of Texier-Zumbrun \cite{TZ}. In order to do so, we consider the same particular values as in their Remark 4.7 and
Examples 4.8, 4.9 pages 517-518. Let us illustrate this in 2 and 3 space dimensions.
\bigskip

 {\bf In two space dimensions}, $d=2$ (Example 4.8 in \cite{TZ}):

Our condition becomes $\frac{1}{p-1}<\kappa$. In their paper,
Texier and Zumbrun need the stronger condition $\frac{9}{2(p+1)}<\kappa$.\medskip

{\bf In three space dimensions}, $d=3$ (Example 4.9 in \cite{TZ}):

Our condition becomes $\frac{3}{2(p-1)}<\kappa$. In their paper,
Texier-Zumbrun need the stronger condition $\frac{4}{p+1}<\kappa$.\medskip

In both cases, we improve over M\'etivier-Rauch when $p\geq 3$, while Texier-Zumbrun need $p\geq 4$.\medskip

{\bf Remark.} {\it After reading our paper, Baldi and Haus \cite{BHperso} have been able to relax even further the condition on $\kappa$, based on their version \cite{BH} of the classical Newton scheme in the spirit of H\"{o}rmander. A key point in their proof is a clever modification of the norms considered by Texier-Zumbrun, allowing better $C^2$ estimates on the functional. They also explain that their approach can be extended to other $C^2$ functionals consisting of a linear term perturbed by a nonlinear term of homogeneity at least $p+1$. Our abstract theorem, however, seems more general since we do not need such a structure.}

\subsection{Proof of Theorem \ref{ThmTexier}}

We have to show that our Corollary \ref{Cor9} applies. Our functional setting is
the same as in \cite{TZ}, with slightly different notations.\medskip

We introduce the norm $\Vert f\Vert_{H_{\varepsilon}^{s}(\mathbf{R}^{d}%
)}=\Vert(-\varepsilon^{2}\Delta+1)^{s/2}f\Vert_{L^{2}(\mathbf{R}^{d})}$, and
we take
\begin{align*}
V_{s}  &  =\mathcal{C}^{1}([0,T],H^{s-2}(\mathbf{R}^{d}))\cap\mathcal{C}^{0}([0,T],H^{s}(\mathbf{R}^{d}))\;,\\
\vert u\vert_{s}  &  =\sup_{0\leq t\leq T}\left\{  \Vert\varepsilon
^{2}\partial_{t}u(t,\cdot)\Vert_{H_{\varepsilon}^{s-2}(\mathbf{R}^{d})}+\Vert
u(t,\cdot)\Vert_{H_{\varepsilon}^{s}(\mathbf{R}^{d})}\right\}  \;
\end{align*}
and
\begin{align*}
W_{s}  &  =\mathcal{C}^{0}([0,T],H^{s}(\mathbf{R}^{d}))\times H^{s+2}
(\mathbf{R}^{d})\;,\\
\vert(v_{1},v_{2})\vert_{s}^{\prime}  &  =\sup_{0\leq t\leq T}\left\{  \Vert
v_{1}(t,\cdot)\Vert_{H_{\varepsilon}^{s}(\mathbf{R}^{d})}\right\}  +\Vert
v_{2}\Vert_{H_{\varepsilon}^{s+2}(\mathbf{R}^{d})}%
\end{align*}

Our projectors are
\begin{align*}
\Pi_{\Lambda}u  &  =\mathcal{F}_{x}^{-1}(1_{|\varepsilon\xi|\leq \Lambda}\mathcal{F}%
_{x}u(t,\xi))\,,\\
\Pi_{\Lambda}^{\prime}(v_{1},v_{2})  &  =\left(  \mathcal{F}_{x}^{-1}%
(1_{|\varepsilon\xi|\leq \Lambda}\mathcal{F}_{x}v_{1}(t,\xi)),\mathcal{F}%
^{-1}(1_{|\varepsilon\xi|\leq \Lambda}\mathcal{F}v_{2}(\xi))\right)
\end{align*}

We take
$$\Phi_\varepsilon(u)  =\left(  \varepsilon^{2}\partial_{t}u+iA(\varepsilon
\partial_{x})u-\varepsilon B(u,\varepsilon\partial_{x})u\,,\,u(0,\cdot)-\varepsilon^{\kappa}(a_{\varepsilon
},\bar{a}_{\varepsilon})\right) $$
and
$${\mathfrak a}_\varepsilon(t,x)  =\varepsilon^{\kappa}(\exp(-itA(\partial_{x}))a_{\varepsilon
},\exp(itA(\partial_{x}))\bar{a}_{\varepsilon})\;.$$

We have $\Phi_{\varepsilon}({\mathfrak a}_\varepsilon)=(-\varepsilon
B({\mathfrak a}_\varepsilon,\varepsilon\partial_{x}){\mathfrak a}_\varepsilon,0)$. A solution of the
functional equation $\Phi_\varepsilon(u)=0$ is a solution
on $[0,T]\times\mathbf{R}^{d}$
of the Cauchy problem \ref{Cauchy}.\medskip

Our Corollary \ref{Cor9} requires a direct estimate
(\ref{tamedirectepsilonbis}) on $D\Phi_\varepsilon$ and an estimate (\ref{tameinverseepsilonbis})
on the right-inverse $L_\varepsilon$.

Take $s_{0}>d/2+2$, $m=2$, $\gamma
=\frac{dp}{2(p-1)}$ and $S$ large. Since $\kappa>\frac{d}{2(p-1)}$, we have
an estimate of the form $\vert {\mathfrak a}_\varepsilon\vert_{S}\lesssim \varepsilon^{\gamma}\Vert a_1\Vert_{H^{S}}$, so, taking $\Vert a_1\Vert_{H^{S}}$ small, we can ensure that
${\mathfrak a}_\varepsilon\in {\mathfrak B}_{S}(\varepsilon^{\gamma})$. Moreover the inequality
$\kappa>\frac{d}{2(p-1)}$ implies the condition
\[1-\frac{dp}{2}+p\gamma\geq 0\,.\]
So we see that the assumptions of Lemma 4.4 in \cite{TZ} are satisfied by the parameters $\gamma_0=\gamma_1=\gamma$ (note that our exponent $p$ is denoted $\ell$ in \cite{TZ}).
The direct estimate (\ref{tamedirectepsilonbis}) thus
follows from Lemma 4.4 in \cite{TZ}. Note that Lemma 4.4 of \cite{TZ} also gives an estimate on the second
derivative of $\Phi_\varepsilon(\cdot)$, but we do not need
such an estimate.\medskip

Choosing, in addition, $\ell=2$, $\ell^{\prime}=0$,  $g=2$, our inverse estimate (\ref{tameinverseepsilonbis}) follows from from Lemma 4.5 in \cite{TZ}.\medskip

To summarize, the assumptions (2.9, 2.10, 2.11) of our Corollary \ref{Cor9} are satisfied for $s_{0}>d/2$, $m=2$, $\gamma
=\frac{p}{p-1}\frac{d}{2}$, $g=2$, $\ell=2$, $\ell^{\prime}=0$.\medskip

Moreover, in \cite{TZ}, {\it Proof of Theorem 4.6}, one finds an estimate which can be written in the form
$$\vert \Phi_{\varepsilon}({\mathfrak a}_\varepsilon)\vert_{s_1-1}^{\prime}\leq r\,\varepsilon^{1+\kappa
(p+1)+d/2}$$
where $r$ is small when $\Vert a_1\Vert_{H^{s_1}}$ is small.\medskip

So, using our Corollary \ref{Cor9}, taking $S$ large enough we can
solve the equation $\Phi_\varepsilon(u)=0$ in $X_{s_1}$ under the additional condition $1+\kappa(p+1)+d/2>\gamma+g\,,$ which can be rewritten as follows:
\[
\kappa>\frac{1}{p+1}+\frac{d}{2(p+1)(p-1)}\,.
\]
Since $d\geq 2$, this inequality is a consequence of our assumption $\kappa\geq \frac{d}{2(p-1)}\,.$\medskip

So our Corollary \ref{Cor9} implies the existence of a solution to the Cauchy problem (\ref{Cauchy}).
The uniqueness of this solution comes from the local-in-time uniqueness of solutions to the Cauchy problem. This proves Theorem \ref{ThmTexier} as a consequence of Corollary \ref{Cor9}.\medskip

{\bf Remark.}
{\it In the examples 4.8 and 4.9 of \cite{TZ}, Texier and Zumbrun also study the case of oscillating initial data, i.e. $a_\varepsilon=a(x) e^{ix\cdot\xi_0/\varepsilon}$, and in the first submitted version of this paper we considered it as well. However, a referee pointed out to us that the corresponding statements were not fully justified in \cite{TZ}. Indeed, in the proof of their Theorem 4.6, Texier and Zumbrun have to invert the linearized functional $D\Phi_\varepsilon(u)$ for $u$ in a neighborhood of the function ${\mathfrak a}_\varepsilon$, denoted $a_f$ in their paper. For this purpose, it seems that they need the norm of their function $a_f$ to be controlled by $\varepsilon^\gamma$. This condition appears in their Remark 2.14 and their Lemma 4.5, but not in the statement of their Theorem 4.6. This additional constraint does not affect their results for concentrating initial data in Examples 4.8, 4.9. But in the oscillating case, their statements seem overly optimistic. We did not want to investigate further that issue, this is why we only deal with the concentrating case. Note, however, that this difficulty with the oscillating case is overcome in the recent work \cite{BHperso}, thanks to improved norms and estimates.}
\bigskip

\section{Conclusion}

The purpose of this paper has been to introduce a new algorithm into the
"hard" inverse function theorem, where both $DF\left(  u\right)  $ and its
right inverse $L\left(  u\right)  $ lose derivatives, in order to improve its
range of validity. To highlight this improvement, we have considered singular perturbation problems with loss of derivatives. We have shown that, on the specific example of a
Schr\"{o}dinger-type system of PDEs arising from nonlinear optics, our method
leads to substantial improvements of known results. We believe that our approach has
the potential of improving the known estimates in many other ``hard" inversion problems.

In the statement and proof of our abstract theorem,
our main focus has been the existence of $u$ solving $F(u)=v$ in the
case when $S$ is large and the regularity of $v$ is as small as possible. We haven't tried to give an explicit
bound on $S$, but with some additional work, it can be done. In an earlier version \cite{ESDebut}
of this paper, the reader will find a study of the intermediate case of a
tame Galerkin right-invertible differential $DF$, with precise estimates on the parameter $S$
depending on the loss of regularity of the right-inverse, in the special case $s_0=m=0$ and $\ell=\ell'$.


\end{document}